\documentclass{article}
\usepackage{amsfonts}
\usepackage{amsmath}
\usepackage{amssymb}
\usepackage{amsthm}
\usepackage{graphicx}
\usepackage{braket}
\usepackage{typearea}
\usepackage{fullpage}
\usepackage[matrix,arrow,curve]{xy}
\usepackage{tikz}
\usetikzlibrary{arrows,shapes,trees,calc,through}

\theoremstyle{definition}
\newtheorem{definition}{Definition}
\newtheorem{lemma}[definition]{Lemma}
\newtheorem{theorem}[definition]{Theorem}

\newtheorem{corollary}[definition]{Corollary}
\newtheorem{example}[definition]{Example}

\newenvironment{@abssec}[1]{%
     \if@twocolumn
       \section*{#1}%
     \else
       \vspace{.05in}\footnotesize
       \parindent .2in
         {\bfseries #1. }\ignorespaces
     \fi}
     {\if@twocolumn\else\par\vspace{.1in}\fi}

\def\eqbd{\mathop{{:}{=}}}

\def\eqbd{\mathop{{:}{=}}}
\def\bdeq{\mathop{{=}{:}}}

\def\Re{\mathop{\rm Re}\nolimits}
\def\openC{{\rm C\kern-.48em\vrule width.06em height.6em depth-.02em
                 \kern.48em}}
\def\openQ{{{\rm Q\kern-.21cm\vrule width.6pt height 6.2ptdepth-.2pt \kern.21cm}}}
\def\openR{{{\rm I}\kern-.16em {\rm R}}}
\def\openZ{{{\rm Z}\kern-.28em{\rm Z}}}
\def\openT{{{\rm T}\kern-.42em {\rm T}}}
\def\openH{{{\rm I}\kern-.16em {\rm H}}}
\def\openK{{{\rm I}\kern-.16em {\rm K}}}
\def\openL{{{\rm I}\kern-.16em {\rm L}}}
\def\openM{{{\rm I}\kern-.16em {\rm M}}}
\def\openN{{{\rm I}\kern-.16em {\rm N}}}
\def\openP{{{\rm I}\kern-.16em {\rm P}}}

\def\eqbd{\mathop{{:}{=}}}

\let\C\openC

\def\proofof#1{\noindent {\bf Proof of #1. \/}}
\def\eop{\hfill
        {\ \vbox{\hrule\hbox{\vrule height1.3ex\hskip1.4ex\vrule}\hrule}}
        \vskip 0.3cm \par}

\def\belowrightarrow#1{{{{}\over\ #1\ }\kern-1.1em\to}}
\def\l2{{L_2}}

\def\eqbd{\mathop{{:}{=}}}
\def\bdeq{\mathop{{=}{:}}}

\def\Re{\mathop{\rm Re}\nolimits}

\def\eqbd{\mathop{:}{=}}

\def\belowrightarrow#1{{{{}\over\ #1\ }\kern-1.1em\to}}
\def\l2{{L_2}}

\begin{document}

\title{ Generalized Hurwitz matrices, generalized Euclidean algorithm, and forbidden
sectors of the complex plane}

\author{ Olga Holtz\footnotemark[1] \footnotemark[2] ,
Sergey Khrushchev\footnotemark[3], and Olga Kushel\footnotemark[4]}

\renewcommand{\thefootnote}{\fnsymbol{footnote}}

\footnotetext[1]{
Department of Mathematics MA 4-2,
Technische Universit\"at Berlin,
Strasse des 17. Juni 136,
D-10623 Berlin,
Germany.}
\footnotetext[2]{
Department of Mathematics,
  University of California-Berkeley,
      821 Evans Hall,
      Berkeley, California, 94720.
Telephone: +1 510 642 2122 
Fax: +1 510 642 8204
Email: holtz@math.berkeley.edu
}
\footnotetext[3]{
International School of Economics,
Kazakh-British University,
Tole bi 59, 050000
Almaty, Kazakhstan.
}
\footnotetext[4]{
Department of Mathematics,
Shanghai Jiao Tong University,
Shanghai 200240, China.
}

\date{\small June 23, 2015}
\maketitle

\begin{abstract} \noindent
Given a polynomial \[ f(x)=a_0x^n+a_1x^{n-1}+\cdots +a_n \] with positive coefficients $a_k$, and a positive integer $M\leq n$, we define a(n infinite) generalized Hurwitz matrix $H_M(f)\eqbd (a_{Mj-i})_{i,j}$. We prove that  the polynomial $f(z)$ does not vanish in the sector 
$$ \left\{z\in\mathbb{C}: |\arg (z)| < \frac{\pi}{M}\right\} $$
whenever the matrix $H_M$ is totally nonnegative. This result generalizes the classical Hurwitz' Theorem on stable polynomials ($M=2$), the Aissen-Edrei-Schoenberg-Whitney theorem on polynomials with negative real roots ($M=1$), and the Cowling-Thron theorem ($M=n$).  In this connection, we also develop a generalization of the classical Euclidean algorithm, of independent interest \emph{per se}.
\end{abstract}


\section*{Introduction}
The problem of determining the number of zeros of a polynomial in a given region of the complex plane is very classical and goes back to Descartes, Gauss,  Cauchy \cite{CA}, Routh \cite{RA1,RA2}, Hermite \cite{HEM}, Hurwitz~\cite{HU}, and many others. The entire second volume of the delightful \emph{Problems and Theorems in Analysis} by P\'olya and Szeg\H{o} \cite{PSII} 
is devoted to this and related problems. See also comprehensive monographs of Marden \cite{MA}, Obreshkoff \cite{OB}, and Fisk \cite{FISK}.

One particularly famous late-19th-century result, which also has numerous applications, is the Routh-Hurwitz criterion of stability. Recall that a polynomial is called \emph{stable} if all its zeros lie in the open left half-plane of the complex plane. The Routh-Hurwitz criterion asserts the following:
\begin{theorem}[Routh-Hurwitz \cite{HU,RA1,RA2}]\label{HurwitzTheorem}
A real polynomial $f(x) = a_0x^n + a_1x^{n-1}+ \cdots + a_n$ $(a_0 > 0)$ is stable if and only if all leading principal minors of its Hurwitz matrix $H_2(f)$ up to order $n$ are positive.
\end{theorem}

Decades after Routh-Hurwitz, Asner \cite{ASN} and Kemperman \cite{KEM} independently realized that the Routh-Hurwitz criterion can be restated in terms of the total nonnegativity of the Hurwitz matrix. Moreover,  the Hurwitz matrix of a stable polynomial admits a simple factorization into totally nonnegative factors \cite{HOL}.  These developments are described in \cite[Section 4.11]{PINK}; see 
also a separate section \cite[Section 4.8]{PINK} on generalized Hurwitz matrices. The converse direction of the total nonnegativity criterion was fully established only a few years ago in \cite{HT}:
\begin{theorem}[\cite{HT}]\label{HurwitzTheoremTN}
A polynomial $f(x) = a_0x^n + a_1x^{n-1}+ \cdots + a_n$ $(a_0,  a_1,  \ldots,  a_n \in {\mathbb R}; \ a_0 > 0)$ has no zeros in the open right half-plane $\Re z> 0$ if and only if its Hurwitz matrix $H_2(f)$ is totally nonnegative.
\end{theorem}
Note that the strict stability has to be replaced here by its natural weaker counterpart. As we shall discuss in Section~\ref{sec:factorization}, generalized Hurwitz matrices turn out to enjoy analogous properties!  (We should add that Pinkus \cite[Section 4.11]{PINK} mistakenly asserts that the total nonnegativity of the infinite Hurwitz matrix is {\it not} sufficient for all zeros to lie in the closed left half-plane. However, it is the finite, i.e., $n{\times}n$, Hurwitz matrix, whose total nonnegativity is not sufficient.)

Another famous mid-20th-century result of Aissen, Edrei, Schoenberg and Whitney \cite{AESW}  concerns a seemingly different class of polynomials, i.e., those with real negative roots. This result admits a strikingly similar formulation to the total nonnegativity Theorem~\ref{HurwitzTheoremTN}:

\begin{theorem}[\cite{AESW}]\label{AESWTheorem}
A polynomial $f(x) = a_0x^n + a_1x^{n-1}+ \cdots + a_n$ $(a_0,  a_1,  \ldots, a_n \in {\mathbb R}; \ a_0 > 0)$ has only real negative zeros if and only if its Toeplitz matrix $H_1(f)$ is totally nonnegative.
\end{theorem}

The reader may wonder why we refer to the Hurwitz and Toeplitz matrices associated to a polynomial $f$ as $H_2(f)$ and $H_1(f)$. As we shall see, both are special cases of generalized Hurwitz matrices, which will be denoted by $H_M(f)$, for $M=1$ and $M=2$.  Note that we take all matrices $H_M$ to be infinite. Fittingly,  our main result generalizes both Routh-Hurwitz and Aissen-Edrei-Schoenberg-Whitney criteria:

\begin{theorem}\label{coolthm}
A polynomial $f(x) = a_0x^n + a_1x^{n-1} + \cdots + a_n$ $(a_0,  a_1,  \ldots,  a_n \in {\mathbb R}, \ a_0>0)$ of degree $n$ has no zeros in the sector $$ \left\{z\in\mathbb{C}: |\arg (z)| < \frac{\pi}{M}\right\} $$ whenever its generalized Hurwitz matrix $H_M(f)$ is totally nonnegative. 
\end{theorem}

While not enjoying the ideal 'if and only if' format, this result is nevertheless beautifully similar to the two classical results we just revisited.  We shall discuss obstacles to the converse statement for $M>2$ in Section~\ref{sec:obstacles}. Section~\ref{sec:genEuclid} is devoted to a generalized Euclidean algorithm, which is crucial to our proof. That algorithm itself should be of independent interest in algebra and analysis. In Section~\ref{sec:continuedfrac}, we shall examine various continued fractions that can be constructed from the generalized Euclidean algorithm. In Section~\ref{sec:genHurwitz} we shall properly introduce generalized Hurwitz matrices and point out their connections with the generalized Euclidean algorithm. We shall then examine the interplay between generalized Hurwitz matrices and their regular Hurwitz submatrices in Section~\ref{sec:submatrices}. The proof of our main theorem will be given in Section~\ref{sec:proof}, along with an interesting factorization of generalized Hurwitz matrices in Section~\ref{sec:factorization}. Various related questions will be discussed in the remaining two sections.
\medskip

Our approach  is based on a systematic study of the following objects:
\begin{itemize}
\item generalized Hurwitz matrices and their submatrices;
\item continued fractions;
\item roots and coefficients of polynomials, especially Routh's (or Euclidean) algorithm.
\end{itemize}
We shall begin by defining and examining the generalized Euclidean algorithm, which will play a crucial role connecting all objects mentioned above.



\section{Generalized Euclidean algorithm}\label{sec:genEuclid}

We will now develop a generalization of Euclidean algorithm for $M>2$ polynomials (replacing $M=2$ for the regular Euclidean algorithm). What folows is in fact a special version designed for the purposes of splitting a given polynomial into $M$ parts according to the residues of the coefficients mod $M$. A more general version of the generalized Euclidean algorithm will be described in another paper.

Let
\[ f(x) = a_0x^n + a_1x^{n-1}+\cdots + a_n \]
 be a polynomial of degree $n$ with real coefficients $a_j$, $j=0, \ldots, n$. As usual, we define $\deg f(x) \eqbd -\infty$ if $f(x)\equiv 0$. Let $M$ be a positive integer, $2 \leq M \leq n$. Then $f$ can be split 
into a sum of polynomials
\begin{equation}\label{svk100}
f(x)=f_0(x)+f_1(x)+\cdots+f_{M-1}(x),
\end{equation}
{\rm where} 
\begin{equation}\label{svk9056}
f_j(x)=\sum_{ l\equiv j\,\text{mod}(M)\atop 0\leq l\leq n} a_lx^{n-l}.
\end{equation}

\begin{definition} A polynomial $p(x)$ is called \emph{arithmetic} with difference $M$ (and 
\emph{residue} $k$) if
\[ p(x)=\sum_{l\equiv k\,\text{mod}(M)} a_lx^{n-l}. \]
\end{definition}
Per this definition, \eqref{svk100} is a decomposition of $f$ into a sum of arithmetic polynomials with difference $M$ and the full set of residues. The degrees of monomials in every $f_j(x)$ form an arithmetic progression with difference $M$; every nonzero monomial term of $f(x)$ enters only one $f_j(x)$. Notice that a zero polynomial is arithmetic for any difference $M$.

A \emph{generalized Euclidean algorithm} associated with the decomposition \eqref{svk100} is defined as follows. For any $i=0,1,\ldots,M-2$, there is a unique representation of the form
\[ f_i(x) = d_i(x)f_{i+1}(x) + f_{i+M}(x), \]
where $d_i(x)$ and $f_{i+M}(x)$ are polynomials, subject to the following rules:
\begin{itemize}
\item[(a)] If $\deg(f_i)\geq\deg(f_{i+1})>-\infty$, then $f_{i+M}$ is the remainder in the division of the polynomial $f_i(x)$ by $f_{i+1}(x)$, and $d_i(x)$ is the quotient. Hence $\deg(f_{i+M})<\deg(f_{i+1})$.

\item[(b)]If $-\infty\leq \deg(f_i)<\deg(f_{i+1})$ then $d_i(x)\equiv 0$ and $f_{i+M}(x)=f_i(x)$.

\item[(c)]If $-\infty=\deg(f_{i+1})$ then $d_i(x)\equiv 0$ and $f_{i+M}(x)=f_i(x)$.
\end{itemize}

The algorithm stops when $f_n(x)$ is constructed.

We call this process the \emph{generalized Euclidean algorithm with step $M$} applied to the polynomials $\{f_0, \ f_1, \ \ldots, \ f_{M-1}\}$. \medskip

\begin{example}
Let $M=3$ and
$
f(x)=x^7+x^6+x^5+x^4+x^3+x^2+x+1.
$
Then
\[
\begin{aligned}f_0(x)&=x^7+x^4+x;\\ f_1(x)&=x^6+x^3+1;\\f_2(x)&=x^5+x^2.\end{aligned}
\]
Next,
\[
\begin{aligned}f_0(x)&=x\cdot f_1(x)+0\\ f_1(x)&=x\cdot f_2(x)+1\\f_2(x)&=0\cdot0+f_2(x)\\ f_3(x)&=0\cdot 1+0\\ f_4(x)&= 0\cdot f_5(x)+1\end{aligned}\quad\Longrightarrow\quad
\begin{aligned}f_3(x)&\equiv 0\\f_4(x)&=1\\f_5(x)&=x^5+x^2\\f_6(x)&\equiv0\\f_7(x)&\equiv 1.\end{aligned}
\]
It is convenient to arrange the resulting polynomials into the following table:
\begin{center}
\begin{tabular}{|l|c|c|c|}  \hline
\textbf{Groups: }$k$&$0$&$1$&$2$\\ \hline
$0+3k$&$x^7+x^4+x$&$0$&$0$\\
$1+3k$&$x^6+x^3+1$&$1$&$1$\\
$2+3k$&$x^5+x^2$&$x^5+x^2$&\text{}\\
\hline
\end{tabular}
\end{center}
\end{example}
\medskip

\begin{example} Let $M=3$ and
$
f(x)=(x+1)^7=x^7+7x^6+21x^5+35x^4+35x^3+21x^2+7x+1.
$
Then
\[
\begin{aligned}f_0(x)&=x^7+35x^4+7x;\\ f_1(x)&=7x^6+35x^3+1;\\f_2(x)&=21x^5+21x^2.\end{aligned}
\]
Next,
\[
\begin{aligned}f_0(x)&=\frac{x}{7}\cdot f_1(x)+30x^4+\frac{48}{7}x\\ f_1(x)&=\frac{x}{3}\cdot f_2(x)+28x^3+1\\f_2(x)&=\frac{7x}{10}f_3(x)+\frac{81x^2}{5}\\ f_3(x)&=\frac{15x}{14}\cdot f_4(x)+\frac{81}{14}x\\ f_4(x)&= \frac{5\cdot 28 x}{81}\cdot f_5(x)+1\end{aligned}\quad\Longrightarrow\quad
\begin{aligned}f_3(x)&=30x^4+\frac{48}{7}x\\f_4(x)&=28x^3+1\\f_5(x)&=\frac{81x^2}{5}\\f_6(x)&=\frac{81}{14}x\\f_7(x)&\equiv 1.\end{aligned}
\]
The table for this polynomial looks as follows:
\begin{center}
\begin{tabular}{|l|c|c|c|}  \hline
\textbf{Groups: }$k$&$0$&$1$&$2$\\ \hline
$0+3k$&$x^7+35x^4+7x$&$30x^4+\frac{48}{7}x$&$\frac{81}{14}x$\\
$1+3k$&$7x^6+35x^3+1$&$28x^3+1$&$1$\\
$2+3k$&$21x^5+21x^2$&$\frac{81x^2}{5}$&\text{}\\
\hline
\end{tabular}
\end{center}
\end{example}
\medskip

\begin{example} Let $M=3$ and
$
f(x)=x^7+x^6+x^5.
$
Then
\[
\begin{aligned}f_0(x)&=x^7;\\ f_1(x)&=x^6;\\f_2(x)&=x^5.\end{aligned}
\]
Next,
\[
\begin{aligned}f_0(x)&=x\cdot f_1(x)+0\\ f_1(x)&=x\cdot f_2(x)+0\\f_2(x)&=0\cdot f_3(x)+x^5\\ f_3(x)&= 0\cdot f_4(x)+0\\ f_4(x)&= 0\cdot f_5(x)+0\end{aligned}\quad\Longrightarrow\quad
\begin{aligned}f_3(x)&\equiv 0\\f_4(x)&=0\\f_5(x)&=x^5\\f_6(x)&\equiv 0\\f_7(x)&\equiv 0.\end{aligned}
\]
The table for this example is the following:
\begin{center}
\begin{tabular}{|l|c|c|c|}  \hline
\textbf{Groups: }$k$&$0$&$1$&$2$\\ \hline
$0+3k$&$x^7$&$0$&$0$\\
$1+3k$&$x^6$&$0$&$0$\\
$2+3k$&$x^5$&$x^5$&\text{}\\
\hline
\end{tabular}
\end{center}
\end{example}
\medskip

\begin{example} Let $M=4$ and
$
f(x)=x^9+x^7+x^6+x^5+x^3+x^2+x.
$
Then the table looks like
\begin{center}
\begin{tabular}{|l|c|c|c|}  \hline
\textbf{Groups: }$k$&$0$&$1$&$2$\\ \hline
$0+4k$&$x^9+x^5+x$&$x^9+x^5+x$&$x^9+x^5+x$\\
$1+4k$&$0$&$0$&$0$\\
$2+4k$&$x^7+x^3$&$0$&\text{}\\
$3+4k$&$x^6+x^2$&$x^6+x^2$&\text{}\\
\hline
\end{tabular}
\end{center}
\end{example}
\medskip

\begin{theorem}Let $f(x)$ be a polynomial of degree $n\geq 2$ with real coefficients and  let $2\leq M\leq n$. Let $\{f_0,f_1,\ldots,f_n\}$ be the polynomials obtained by the generalized Euclidean algorithm with step $M$. Then
\begin{itemize}
\item[A.] Every nonzero $f_i(x)$ is arithmetic with difference $M$.
\item[B.] If both polynomials $f_i(x)$ and $f_{i+M}(x)$ are nonzero, then their residues are equal.
\item[C.] For every $i$ either $d_i(x)\equiv 0$ or $d_i(x)$ is an arithmetic polynomial whose residue equals the difference between the residues of $f_i$ and $f_{i+1}$. In particular its degree is greater or equal to that difference.
\item[D.] If $f_i(x)$ and $f_{i+1}(x)$ are nonzero polynomials, then
\begin{equation}\label{svk101}|\deg(f_i)-\deg(f_{i+1})|\geq 1.
\end{equation}
\end{itemize}
\end{theorem}
\begin{proof} Consider the table of $f(x)$ associated with the generalized Euclidian algorithm with step $M$:
\begin{center}
\begin{tabular}{|l|c|c|c|c|c|}  \hline
\textbf{Groups: }$k$&$k=0$&$k=1$&$k=2$&$\cdots$&$k=$\textbf{last}\\ \hline
$0+Mk$&$f_0(x)$&$f_M(x)$&$f_{2M}(x)$&$\cdots$&$f_{n-j}(x)$\\
$1+Mk$&$f_1(x)$&$f_{1+M}(x)$&$f_{1+2M}(x)$&$\cdots$&$f_{n-j+1}(x)$\\
$\vdots$&$\vdots$&$\vdots$&$\vdots$&$\cdots$&$\vdots$\\
$j+Mk$&$f_j(x)$&$f_{j+M}(x)$&$f_{j+2M}(x)$&$\cdots$&$f_n(x)$\\
$\vdots$&$\vdots$&$\vdots$&$\vdots$&$\cdots$&\text{}\\
$M(k+1)-1$&$f_{M-1}(x)$&$f_{M-1+M}(x)$&$f_{M-1+2M}(x)$&$\cdots$&\text{}\\ \hline
\end{tabular}
\end{center}

By the definition of the generalized Euclidean algorithm, all nonzero $f_i(x)$ with $0\leq i\leq M-1$ in group $0$ are arithmetic with difference $M$ and distinct residues. It follows that no pair of them can have equal degrees, implying \eqref{svk101} for $i=0,\ldots,i=M-2$.

Group $1$ of the polynomials $f_{i+M}$, $i=0,\ldots, M-1$ (see column $1$ in the table above) is determined by the identities
\[
f_i(x) = d_i(x)f_{i+1}(x) + f_{i+M}(x),\quad i=0,1,\ldots, M-1.
\]
If $0\leq i\leq M-2$ and we have the case (a) of the algorithm, then $f_{i+M}$ is the remainder in the division of one arithmetic polynomial, $f_i(x)$, by another arithmetic polynomial, $f_{i+1}(x)$ (with a different residue). It follows that $d_i(x)$ is an arithmetic polynomial with residue $k>0$, which shifts the arithmetic progression of the exponents in $f_{i+1}(x)$ to the arithmetic progression of the exponents of $f_i(x)$.  Hence $f_{i+M}=f_i(x)-d_i(x)f_{i+1}(x)$ is either the zero polynomial or an arithmetic polynomial with the same residue as $f_i(x)$ and $\deg(f_{i+M})<\deg(f_{i+1})<\deg(f_i)$.

In cases (b) and (c) we have $f_{i+M}\equiv f_i$. Moreover, if $f_i(x)\equiv 0$, then by requirement (b) of the algorithm all polynomials $f_{i+kM}$, $k=1,\ldots$, are zero too. If $f_{i+1}(x)\equiv 0$, then by requirement (c) of the algorithm $f_i\equiv f_{i+M}$. Since $f_{i+M+1}\equiv 0$ by our argument above, we get $f_i=f_{i+M}=f_{i+2M}$. Continuing by induction, we see that the whole row starting with $f_i$ is filled with $f_i$.

As to the 'boundary' pair $f_{M-1}$ and $f_M$, in case both of them are nonzero, the residue of $f_M$ equals the residue of $f_0$, which is $n$, whereas the residue of $f_{M-1}$ is $n-M+1$. Since $n-(n-M+1)=M-1$, we see that polynomials $f_{M-1}$ and $f_M$ have different residues, implying \eqref{svk101} in this case too.

The process can be continued by induction on $k$. In this way, we obtain two possibilities for each row: In the first case we have polynomials of strictly decreasing degrees and the row terminates either in a string of zero polynomials or in several copies of the same polynomial. In the second case the row is simply filled with zeros. In the third case it is filled with copies of the same polynomial.
\end{proof}
\begin{corollary}\label{corGEA}Let $f(x)$ be a polynomial of degree $n\geq 2$ with real coefficients and let $2\leq M\leq n$. Let $\{f_0,f_1,\ldots,f_n\}$ be the polynomials obtained by the generalized Euclidean algorithm with step $M$. If none of these polynomials is zero, then $d_i(x)=c_ix$ for $i=0,1,\ldots, n-1$ and $\deg(f_k)=n-k$.
\end{corollary}

\noindent
{\bf Remark.} We will refer to the situation of Corollary~\ref{corGEA} where none of the polynomials $\{ f_0, f_1 ,\ldots f_n\}$ is zero as the \emph{non-degenerate case}  of the generalized Euclidean algorithm.


\section{Continued fraction expansions}\label{sec:continuedfrac}
Of course, the regular Euclidean algorithm can be also applied to any pair of polynomials $(f_i,f_j)$ generated by the generalized Euclidean algorithm. Let us look into this, assuming the non-degenerate case. 

Suppose $0\leq i<j<M$. Denote the fraction $\dfrac{f_i}{f_j}$ by $R_{ij}$. We want to represent the function $R_{ij}(z)$ as a continued fraction
\begin{equation}\label{ku4}
R_{ij}(z) = q_1^{ij}(z)+ \dfrac{1}{q_2^{ij}(z)+ \dfrac{1}{\ddots+ \dfrac{1}{q_k^{ij}(z)}}}.
\end{equation}
Applying the ordinary Euclidean algorithm to the pair $(f_i,  f_j)$, we construct a sequence of polynomials $f^{ij}_0, \ f^{ij}_1, \ \ldots, \ f^{ij}_k$ with leading coefficients $h^{ij}_0, \ h^{ij}_1, \ \ldots, \ h^{ij}_k$, respectively, as follows:
\begin{eqnarray*}
f^{ij}_0(x) & := & f_i(x) = a_ix^{n-i} + a_{i+M}x^{n-i-M} + a_{i+2M}x^{n- i-2M} + \cdots  \\
f^{ij}_1(x)  & :=  & f_j(x) = a_jx^{n-j} + a_{j+M}x^{n-j-M} + a_{j+2M}x^{n- j-2M} + \cdots
\end{eqnarray*}
The subsequent polynomials are defined by extracting the leading term from the ratios
$\dfrac{f^{ij}_{\ell-1}}{f^{ij}_{\ell}}$:
$$f^{ij}_{\ell-1}(x) = : q^{ij}_{\ell}(x) f^{ij}_{\ell}(x) + f^{ij}_{\ell+1}(x),   \qquad \qquad \ell=1, 2, \ldots$$  where  $$q^{ij}_\ell(x) := \left\{\begin{array}{ll}
\dfrac{h^{ij}_{\ell-1}}{h^{ij}_\ell} \, x^{M -(j-i)} & \mbox{if $\ell$ is even};\\[10pt]
\dfrac{h^{ij}_{\ell-1}}{h^{ij}_\ell} \, x^{j-i} & \mbox{if $\ell$ is odd,} \end{array}\right.$$
$$\deg f^{ij}_{l+1}(x) = n - j -M \left\lceil \frac{\ell}{2} \right\rceil.$$
Hence the continued fraction \eqref{ku4} can be written explicitly as
\begin{equation}\label{ku5}
R_{ij}(z) = \dfrac{f_i(z)}{f_j(z)} = \dfrac{h^{ij}_{0}}{h^{ij}_1}z^{j-i}+ \dfrac{1}{\dfrac{h^{ij}_{1}}{h^{ij}_2}z^{M -(j-i)}+ \dfrac{1}{\ddots+\dfrac{1}{\dfrac{h^{ij}_{k-1}}{h^{ij}_k}z^{\mu}}}},
\end{equation}
where $$\mu =\left\{\begin{array}{ll}
M-(j-i)& \mbox{if $k$ is even}  ;\\[10pt]
j-i& \mbox{if $k$ is odd.} \end{array}\right.$$



We now proceed to make a simple but crucial observation about the continued fractions $R_{ij}$.
It turns out that these fractions, viewed as functions on $\C$, map cones with sufficiently small apertures emanating from the origin to similar cones. Here are the details.

Given two angles $\alpha<\beta$ between $-\pi$ and $\pi$, consider the cone
\begin{equation*}
K_{\alpha, \beta} := \{z \in {\mathbb C}: \alpha \leq \arg(z) \leq \beta \}.  \label{cones}
\end{equation*}

\begin{lemma}\label{lem61}
Let $R$ be a continued fraction
$$ R(z)= {a_1}z^m +\dfrac{1}{a_2 z^{M-m} +\dfrac{1}{a_3 z^m +\dfrac{1}{\ddots +\dfrac{1}{a_k z^\mu}}}}, $$ with all coefficients $a_j$ positive, $j=1,\ldots, k$, with $m$ an integer between $0$ and
$M$, and with the exponents alternating between $m$ and $M-m$ (so that $\mu=m$ if $k$ is odd and $\mu=M-m$ if $k$ is even). Then, for any angle $\alpha\in [0,\pi/M]$, the function $R$ maps the cone $K_{0,\alpha}$ into the cone $K_{-(M-m)\alpha,m\alpha}$.
\end{lemma}

\begin{proof}  The map $R$ is a composition of special monomial maps, multiplication by constants, additions, and inversion. Let us examine how these maps act on our cone $K_{0,\alpha}$.

A function $()^j : z\mapsto z^j$ maps a cone $K_{\alpha,\beta}$ to $K_{j\alpha,j\beta}$. Multiplication by a positive constant leaves any such cone invariant.  Inversion $()^{-1}$ maps a cone $K_{\alpha,\beta}$ (of course excluding the origin) to the cone $K_{-\beta,-\alpha}$. Equipped with these basic observations, we can now understand the action of $R$ on a cone of type $K_{0,\alpha}$ for $\alpha\leq \pi/M$.

Suppose for simplicity that $k$ is odd. Then the last monomial, $()^m$, maps $K_{0,\alpha}$ to $K_{0,m\alpha}$. Multiplication by $a_k$ leaves the latter cone invariant, and inversion maps it to the cone $K_{-m\alpha, 0}$. The previous monomial (followed by muplication by a positive constant $a_{k-1}$) maps $K_{0,\alpha}$ to $K_{0,(M-m)\alpha}$, so the result lies in
 $K_{0,(M-m)\alpha} + K_{-m\alpha,0} \subset K_{-m\alpha, (M-m)\alpha}$.

The inversion that follows maps $K_{-m\alpha, (M-m)\alpha}$ to $K_{-(M-m)\alpha,m\alpha}$. The function $a_{k-2}()^m$ maps $K_{0,\alpha}$ to $K_{0,m\alpha}$, and the next cone addition yields
$ K_{0,m\alpha} + K_{-(M-m)\alpha, m\alpha} \subset K_{-(M-m)\alpha,m\alpha}$.
From this point on, we shall alternate between the cone $K_{-(M-m)\alpha, m\alpha}$ and its
reflection $K_{-m\alpha,(M-m)\alpha}$ since
\begin{eqnarray*}
K_{0,(M-m)\alpha} + (K_{-(M-m)\alpha,m\alpha})^{-1} & \subset & K_{-m\alpha, (M-m)\alpha}   \\
K_{0,m\alpha} + (K_{-m\alpha,(M-m)\alpha})^{-1} & \subset & K_{-(M-m)\alpha, m\alpha}.
\end{eqnarray*}
Our last operation is of the second type, so we shall end up inside the cone $K_{-(M-m)\alpha, m\alpha}$. Note that this happens regardless of the parity of $k$.
\hfill \qed \end{proof} \medskip

As promised above, we can now apply this lemma to our functions $R_{ij}$:

\begin{corollary}\label{cor22}
Let $R_{ij} = \dfrac{f_i}{f_j}$ $(0 \leq i < j \leq M-1)$ be a rational function defined by continued fraction \eqref{ku5} with all coefficients $h_r$, $r = 0, \ \ldots, \ k$ positive.  Then
 $R_{ij}$ maps the cone $K_{0,\alpha}$ into the cone $K_{-(M-j+i)\alpha,(j-i)\alpha}$ whenever $0\leq \alpha\leq \pi/M$.

\end{corollary}


\section{Generalized Hurwitz matrices}\label{sec:genHurwitz}

Every polynomial
\[
f(x) = a_0x^n + a_1x^{n-1}+\cdots + a_n
\]
and an integer $2\leq M\leq n$ determine a \textit{generalized Hurwitz matrix}~(defined by Goldman and Sun in \cite{GoSu})
\[
H_M=( a_{Mj-i} )_{i,j=1}^{\infty} .
\]
In the formula for the entries of $H_M$ we assume that $a_k=0$ for integers $k<0$ and $k>n$. Thus a generalized Hurwitz matrix is an infinite matrix of the following form:
\[
H_M=\begin{pmatrix}a_{M-1}&a_{2M-1}&a_{3M-1}&\cdots\\a_{M-2}&a_{2M-2}&a_{3M-2}&\cdots\\\vdots&\vdots&\vdots&\\a_0&a_M&a_{2M}&\cdots\\
0&a_{M-1}&a_{2M-1}&\cdots\\0&a_{M-2}&a_{2M-2}&\cdots\\\vdots&\vdots&\vdots&\\
0&a_0&a_M&\cdots\\0&0&a_{M-1}&\cdots\\\vdots&\vdots&\vdots& \ddots \end{pmatrix}
\]
The matrix $H_M$ is constructed from the coefficients of the arithmetic polynomials $f_{M-1}$, $f_{M-2}$, $\ldots$, $f_0$ of the polynomial $f$. The first row is filled with the coefficients of $f_{M-1}$. The second row is filled with the coefficients of $f_{M-2}$, etc. The coefficients of $f_0$ form the $M$th row of $H_M$. Then this first block of $M$ rows is shifted one step to the right and placed underneath. And so on.
Note that this structure generalizes both Toeplitz and Hurwitz structures simultaneously.

\begin{definition} A matrix $A$ of size $n\times m$, where $n$ and $m$ may take infinite values, is called totally nonnegative if all its minors are nonnegative:
\[
A\begin{pmatrix}i_1&\ldots&i_p\\j_1&\ldots&j_p\end{pmatrix}\geq 0.
\]
\end{definition}
\noindent Taking $p=1$, we see that all entries of a totally nonnegative matrix are nonnegative.

It was shown in \cite{GoSu} that the generalized Hurwitz matrix $H_M$ is totally nonnegative if and only if its $n$ \emph{special} minors are nonnegative:
\begin{equation}\label{GoSo1}
H_M(k,r)\eqbd H_M\begin{pmatrix}k & k+1 & \ldots & k+r-1 \\
 1& 2 & \ldots & r  \end{pmatrix} \geq 0,
\end{equation}
for $k=1,\ldots M-1$ and $r=1,\ldots,\left\lfloor\frac{n+k-1}{M-1}\right\rfloor$. This enumeration of special minors was introduced by Pinkus in \cite[Theorem 4.6, p.111]{PINK}. 

Another natural enumeration was introduced by Goldman and Sun in \cite{GoSu}. For this enumeration, we need to 
locate the bottom left matrix entry $a_p$ of the minor \eqref{GoSo1}  ($i=k+r-1$ and $j=r$):
\begin{equation}\label{GoSo2}
p=Mj-i=(M-1)r-(k-1)=(M-1)(r-1)+(M-k).
\end{equation}
Every number $p=1,2,\ldots,n$ can be uniquely represented in the form \eqref{GoSo2}. Therefore if we are given such a $p$, we can determine $k$ and $r$. We denote by
\[
\Delta_p\eqbd H_M(k,r)=H_M\begin{pmatrix}k & k+1 & \ldots & k+r-1 \\
 1& 2 & \ldots & r  \end{pmatrix}
\]
the corresponding minor of the generalized Hurwitz matrix.
\begin{theorem}[{\cite[Theorem 2.1]{GoSu}}]\label{Theorem2.2} Suppose that all special minors
\begin{equation}\label{DeltaCond}
\Delta_p>0,\qquad p=1,\ldots, n
\end{equation}
are positive. Then
\begin{itemize}
\item[$\mathbf{(a)}$] the matrix $H_M$ is totally nonnegative;
\item[$\mathbf{(b)}$] all coefficients $a_p$ of the polynomial $f(x)$ are positive;
\item[$\mathbf{(c)}$] a minor $M$ of $H_M$ is strictly positive if and only if its diagonal elements are strictly positive.
\end{itemize}
\end{theorem}

Our next theorem shows how generalized Hurwitz matrices are related to the generalized Euclidean algorithm.

\begin{theorem}\label{EuHur}
 If all the leading coefficients $\{h_0, \ h_1, \ \ldots, \ h_n\}$ of the polynomials $\{f_0, \ f_1, \ \ldots, \ f_n\}$ are positive, then the conditions $\eqref{DeltaCond}$ are satisfied and the matrix $H_M$ is totaly nonnegative.
\end{theorem}
\begin{corollary}If all the leading coefficients $\{h_0, \ h_1, \ \ldots, \ h_n\}$ of the polynomials $\{f_0, \ f_1, \ \ldots, \ f_n\}$ are positive, then all coefficients of the polynomial $f$ are positive.
\end{corollary}

The proof of Theorem \ref{EuHur} is an easy consequence of Theorem \ref{Theorem2.2} and the following lemma.

\begin{lemma}\label{lem31} Let $a_0 > 0$ and let $h_i$ denote the leading coefficients of the polynomials $f_i$ in the generalized Euclidean algorithm with step $M$. Then
\begin{equation}\label{svk110}
\begin{aligned}
H_M(1,r)  &= \ h_{M-1}h_{2M-2} \ldots h_{rM - r}\\[1mm]
H_M(2,r)  &= \ h_{M-2}h_{2M-3} \ldots h_{rM - (r+1)}\\[1mm]
\cdots\cdots\cdots&  \cdots\cdots\cdots\cdots\cdots\cdots\cdots\cdots\cdots\\[1mm]
H_M(M-1,r)&= \ h_1 h_M \ldots h_{rM - (r+ M-2)},
\end{aligned}
\end{equation}
where $r = 1, \ldots, \left\lceil\dfrac{n}{M-1}\right\rceil$, $h_i:=0$ for $i > n$.\end{lemma}

\begin{proof} 
The matrix $H_M$ is made of the coefficients of polynomials $f_0$, $f_1$, $\ldots$, $f_{M-1}$, arranged into shifted blocks:
$$H_M = \begin{pmatrix} a_{M-1} & a_{2M-1} & a_{3M-1} & \ldots \\
a_{M-2} & a_{2M-2} & a_{3M-2} & \ldots \\
\vdots & \vdots & \vdots & \ddots \\
a_1&a_{M+1}&a_{2M+1}&\ldots\\
a_{0} & a_{M} & a_{2M} & \cdots \\ \vdots & \vdots & \vdots & \ddots
\end{pmatrix}. $$
By \eqref{GoSo2} there are $M-1$ special minors of order $r=1$. Their values are
\[
h_{M-1}=a_{M-1}=H_M(1,1)>0,\qquad \ldots, \qquad h_1=a_1=H_M(M-1,1)>0,
\]
as is claimed by \eqref{svk110} for $r=1$.

There are $M-1$ special minors of order $2$ corresponding to the following matrices:
\[
\begin{pmatrix}a_{M-1}&a_{2M-1}\\a_{M-2}&a_{2M-2}\end{pmatrix},\;\begin{pmatrix}a_{M-2}&a_{2M-2}\\a_{M-3}&a_{2M-3}\end{pmatrix},\ldots,\begin{pmatrix}a_{1}&a_{M+1}\\a_{0}&a_{M}\end{pmatrix}.
\]
To evaluate the determinants of these matrices we apply Gauss elimination. Excluding $a_0$ in the last matrix using $a_1$ from the first row, we obtain a row equivalent matrix:
\[
\begin{pmatrix}a_{1}&a_{M+1}\\a_{0}&a_{M}\end{pmatrix}\sim \begin{pmatrix}a_{1}&a_{M+1}\\0&h_{M}\end{pmatrix}.
\]
Indeed
\[
\begin{aligned}
f_0(x)&=a_0x^n+a_Mx^{n-M}+a_{2M}x^{n-2M}+\cdots;\\
f_1(x)&=a_1x^{n-1}+a_{M+1}x^{n-M-1}+a_{2M+1}x^{n-2M-1}+\cdots.
\end{aligned}
\]
Since $f_0(x)=\frac{a_0}{a_1}xf_1(x)+f_M(x)$, we obtain that
\[
f_M(x)=\left(a_M-\frac{a_0}{a_1}a_{M+1}\right)x^{n-M}+\left(a_{2M}-\frac{a_0}{a_1}a_{2M+1}\right)x^{n-2M}+\cdots.
\]
To summarize, this elimination results in annihilation of $a_0$, $a_0\rightarrow 0$, and in the replacement of all other coefficients in the row of $a_0$ by the coefficients of the polynomial $f_M$. In particular, $a_M\rightarrow h_M$.
Similarly,
\[
\begin{pmatrix}a_{2}&a_{M+2}\\a_{1}&a_{M+1}\end{pmatrix}\sim \begin{pmatrix}a_{2}&a_{M+2}\\0&h_{M+1}\end{pmatrix},\;\ldots\;,
\begin{pmatrix}a_{M-1}&a_{2M-1}\\a_{M-2}&a_{2M-2}\end{pmatrix}\sim \begin{pmatrix}a_{M-1}&a_{2M-1}\\0&h_{2M-2}\end{pmatrix}.
\]
In other words, $a_k\rightarrow 0$ and $a_{M+k}\rightarrow h_{M+k}$.
This proves  \eqref{svk110} for $r=2$.

Let $r>2$ and $1\leq k\leq M-1$. The minor $H(k,r)$ is the determinant of the square matrix
\[
\begin{pmatrix} a_{M-k}& a_{2M-k}&a_{3M-k}&\cdots&a_{rM-k}\\
\vdots&\vdots&\vdots&\ddots&\vdots\\
a_{1} & a_{M+1} & a_{2M+1} & \cdots& a_{(r-1)M+1}\\
a_{0} & a_{M} & a_{2M} & \cdots &a_{(r-1)M}\\
0&a_{M-1}&a_{2M-1}&\cdots&a_{(r-1)M-1}\\
\vdots & \vdots & \vdots & \ddots &\vdots\\
\end{pmatrix}.
\]
Eliminating $a_0$ from the first column using $a_1$, then $a_1$ using $a_2$, and finally $a_{M-k-1}$ using $a_{M-k}$, we see that the above matrix is row equivalent to the matrix
\[
\begin{pmatrix} a_{M-k}& a_{2M-k}&a_{3M-k}&\cdots&a_{rM-k}\\
0&h_{2M-k-1}&h^*_{3M-k-1}&\cdots&h^*_{rM-k-1}\\
\vdots&\vdots&\vdots&\ddots&\vdots\\
0 & h_{M+1} & h^*_{2M+1} & \cdots& h^*_{(r-1)M+1}\\
0 & h_{M} & h^*_{2M} & \cdots &h^*_{(r-1)M}\\
0&a_{M-1}&a_{2M-1}&\cdots&a_{(r-1)M-1}\\
\vdots & \vdots & \vdots & \ddots &\vdots\\
\end{pmatrix},
\]
where the stars denote the coefficients of the polynomials  $f_{2M-k-1}$,$\dots$, $f_{M+1}$, $f_M$. 

Observe that the second column of this matrix ends either with zeros or with some $a_j$, $M-1\geq j\geq 0$. The rows of the matrix with the same indices make up the sequence of the coefficients of the polynomials $f_j$, $\ldots$, $f_{M-1}$, $\ldots$, $f_{2M-k-1}$. It follows that we can run the same elimination process in column two as we did already in column one. As a result, we see that our matrix is row equivalent to
\[
\begin{pmatrix} a_{M-k}& a_{2M-k}&a_{3M-k}&\cdots&a_{rM-k}\\
0&h_{2M-k-1}&h^*_{3M-k-1}&\cdots&h^*_{rM-k-1}\\
0&0&h_{3M-k-2}&\cdots&*\\
\vdots&\vdots&\vdots&\ddots&\vdots\\
0 & 0 & h_{2M+1} & \cdots& *\\
0 & 0 & h_{2M} & \cdots &*\\
0&0&h_{2M-1}&\cdots&*\\
\vdots & \vdots & \vdots & \ddots &\vdots\\
\end{pmatrix}.
\]
Thus the elimination can be continued until we obtain a diagonal matrix which is row equivalent to the initial matrix:
\[
\begin{pmatrix} a_{M-k}& a_{2M-k}&a_{3M-k}&\cdots&a_{rM-k}\\
0&h_{2M-k-1}&h^*_{3M-k-1}&\cdots&h^*_{rM-k-1}\\
0&0&h_{3M-k-2}&\cdots&*\\
\vdots&\vdots&\vdots&\ddots&\vdots\\
0 & 0 & 0 & \cdots& *\\
0 & 0 & 0 & \cdots &*\\
0&0&0&\cdots&*\\
\vdots & \vdots & \vdots & \ddots &\vdots\\
0&0&0&\cdots&h_{rM-k-(r-1)}\\
\end{pmatrix}.
\]
It follows that
\[
H_M(k,r)=h_{M-k}\cdot h_{2M-k-1}\cdots h_{rM-k-1},
\]
which proves the lemma.
\end{proof}
\begin{corollary}\label{cor31}
Let $a_0 > 0$. Then  the leading coefficients of the polynomials
$f_i$, $i = 1, \ \ldots, \ n$, satisfy
$$h_i = \dfrac{H_M(k+1,r)}{H_M(k+1,r-1)}, $$
where $r = \left\lceil\dfrac{i}{M-1}\right\rceil$, $k = r(M-1) - i$, $H_M(k+1,0):=1$.
\end{corollary}
\begin{proof} Applying Lemma~\ref{lem31}, we obtain:
\begin{align*} h_{rM-r} & =  \ \dfrac{H_M(1,r)}{H_M(1,r-1)} \\
h_{rM-(r+1)} & = \ \dfrac{H_M(2,r)}{H_M(2,r-1)} \\
 \cdots \cdots \cdots \cdots & \cdots \cdots  \cdots \cdots \cdots \cdots   \\
h_{rM-(r+M-2)} & = \ \dfrac{H_M(M-1,r)}{H_M(M-1,r-1)}. 
\end{align*}
Thus $$\ \  h_{rM-(r+k)} \  =  \ \dfrac{H_M(k+1,r)}{H_M(k+1,r-1)},$$ where $0 \leq k \leq M-2$.
Assuming $i = rM-(r+k)$, we obtain $r = \left\lceil\dfrac{i}{M-1}\right\rceil$ and $k = r(M-1) - i$.
 \hfill \qed \end{proof} \medskip

\begin{example} Let us consider an example for the case $n=6$, $M=3$. Then
\[
 f(x) = a_0x^6 + a_1x^5+a_2x^4+a_3x^3+a_4x^2 + a_5x+ a_6.
\]
The generalized Hurwitz matrix looks like
\[
H_3 = \begin{pmatrix} a_2 & a_5 & 0&0&\ldots \\
a_1 & a_4 & 0&0&\ldots \\
a_0 & a_3 & a_6&0&\ldots \\
0 & a_2 & a_5&0&\ldots \\
\vdots&\vdots&\vdots&\vdots&\ddots\\
\end{pmatrix}
\]
There are six special minors:
$$ H_3(1,1) = a_2, \ H_3(1,2) = \begin{vmatrix} a_2 & a_5 \\
a_1 & a_4 \end{vmatrix}, \ H_3(1,3) = \begin{vmatrix} a_2 & a_5 & 0 \\a_1 & a_4 & 0 \\ a_0 & a_3 & a_6 \\ \end{vmatrix}.$$
$$ H_3(2,1) = a_1, \ H_3(2,2) = \begin{vmatrix} a_1 & a_4 \\
a_0 & a_3 \end{vmatrix}, \ H_3(2,3) = \begin{vmatrix} a_1 & a_4 & 0 \\a_0 & a_3 & a_6 \\ 0 & a_2 & a_5 \\ \end{vmatrix}.$$

By the generalized Euclidean algorithm,
\[
\begin{aligned}
f_0(x)&=a_0x^6 + a_3x^3 + a_6\\
f_1(x)&= a_1x^5 + a_4x^2\\
f_2(x)&=a_2x^4 + a_5x
\end{aligned}\quad \Longrightarrow \quad\begin{aligned}f_3(x)&=\left(a_3-\frac{a_0a_4}{a_1}\right)x^3+a_6\\
f_4(x)&=\left(a_4-\frac{a_1a_5}{a_2}\right)x^2\\
f_5(x)&=\left(a_5-\frac{a_1a_2a_6}{a_1a_3-a_0a_4}\right)x\\
f_6(x)&=a_6.
\end{aligned}
\]
We have:
\[
h_0=a_0,\; h_1=a_1=H_3(2,1),\; h_2=a_2=H_3(1,1);
\]
\[
h_3 = \dfrac{1}{a_1}(a_1a_3 - a_0a_4) = \dfrac{1}{a_1}\begin{vmatrix} a_1 & a_4 \\
a_0 & a_3 \end{vmatrix} = \dfrac{H_3(2,2)}{H_3(2,1)};
\]
\[
h_4 = \dfrac{1}{a_2}(a_2a_4 - a_1a_5) = \dfrac{1}{a_2}\begin{vmatrix} a_2 & a_5 \\
a_1 & a_4 \end{vmatrix} = \dfrac{H_3(1,2)}{H_3(1,1)};
\]
\[
h_5 = \dfrac{1}{a_1} \dfrac{1}{h_3}(a_5a_3a_1 - a_5a_0a_4 - a_1a_2a_6) = \dfrac{1}{a_1}\dfrac{1}{h_3}\begin{vmatrix} a_1 & a_4 & 0 \\a_0 & a_3 & a_6 \\ 0 & a_2 & a_5 \\ \end{vmatrix} = \dfrac{H_3(2,3)}{H_3(2,2)};
\]
\[
h_6 = a_6 = \dfrac{H_3(1,3)}{H_3(1,2)}.
\]
\end{example}


\section{Submatrices of generalized Hurwitz matrices}\label{sec:submatrices}
Let $H_M(f)$ be a generalized Hurwitz matrix associated to a polynomial
\[
f(x) = a_0x^n + a_1x^{n-1} + \cdots + a_n.
\]
We denote by $H_M^{(ij)}$ its infinite submatrix determined by two polynomials $f_i$ and $f_j$, $0 \leq i < j \leq M-1$:
$$ H_M^{(ij)} = \begin{pmatrix}
a_j & a_{M+j} & a_{2M+j} & a_{3M+j} & \ldots \\
a_i & a_{M+i} & a_{2M+i} & a_{3M+i} & \ldots \\
0 & a_j & a_{M+j} & a_{2M+j} & \ldots \\
0 & a_i & a_{M+i} & a_{2M+i} & \ldots \\
\vdots & \vdots & \vdots & \vdots & \ddots 
\end{pmatrix}.$$
This matrix is the (ordinary) Hurwitz matrix of the polynomial
\[
P^{(ij)}(x)=\sum_{k\geq 0}(a_{kM+i}x^{m-2k}+a_{kM+j}x^{m-2k-1}).
\]
Its arithmetic polynomials for $M=2$ are given by
\begin{equation}\label{s456723}
P^{(ij)}_0(x)=\sum_{k\geq 0}a_{kM+i}x^{m-2k},\quad P^{(ij)}_1(x)=\sum_{k\geq 0}a_{kM+j}x^{m-2k-1}.
\end{equation}

\begin{lemma}\label{LemPtoF}  The even and odd parts of the polynomial $P^{(ij)}(x)$ satisfy
\[
P^{(ij)}_0\left(x^M\right)=x^{m\cdot M-2(n-i)}f_i\left(x^2\right),\quad P^{(ij)}_1\left(x^M\right)=x^{(m-1)\cdot M-2(n-j)}f_j\left(x^2\right).
\]
\end{lemma}
\begin{proof} By \eqref{svk9056} and \eqref{s456723} for $P^{(ij)}_0\left(x^M\right)$ we have
\[
P^{(ij)}_0(x^M)=x^{m\cdot M-2(n-i)}\sum_{k\geq 0}\,a_{kM+i}x^{2n-2(kM+i)}=
x^{m\cdot M-2(n-i)}f_i(x^2).
\]
Similarly,
\[ P^{(ij)}_1(x^M)=x^{(m-1)\cdot M-2(n-j)}\sum_{k\geq 0}\,a_{kM+j}x^{2n-2(kM+j)}=
x^{(m-1)\cdot M-2(n-j)}f_j(x^2), \]
as claimed. \hfill \qed
\end{proof} 

\begin{lemma}\label{degreeformula} Let $f(x) = f_0(x)+f_1(x)+ \cdots + f_{M-1}(x)$ and let all the coefficients $h_0, \ \ldots, \ h_n$ in the generalized Euclidean algorithm with step $M$ be positive.
The degree $m$ of the polynomial $P^{(ij)}(x)$ whose Hurwitz matrix is $H_M^{(ij)}$ is given by
\[
m=\begin{cases}2\Big\lfloor\dfrac{n-i}{M}\Big\rfloor\,&\,\text{ if }\,\Big\lfloor\dfrac{n-i}{M}\Big\rfloor=\Big\lceil\dfrac{n-j}{M}\Big\rceil\\[2mm]
2\Big\lfloor\dfrac{n-i}{M}\Big\rfloor+1\,&\,\text{ if }\,\Big\lfloor\dfrac{n-i}{M}\Big\rfloor=\Big\lfloor\dfrac{n-j}{M}\Big\rfloor\,.  \end{cases}
\]
\end{lemma}
\begin{proof} To find a formula for $m$ we observe that
$H_M^{(ij)}$ is an
infinite-dimensional submatrix of $H_M$ which is obtained from $H_M$ by keeping only rows corresponding to the coefficients of $f_i$ and $f_j$. Thus, the principal minors of $H_M^{(ij)}$ coincide with some minors of $H_M$.
By Theorem \ref{Theorem2.2} (c), the leading principal minors of $H_M^{(ij)}$ are strictly positive if and only if their diagonal elements do not vanish. The diagonal elements of $H_M^{(ij)}$ are:
\begin{equation*}
a_j,\;a_{M+i},\;a_{M+j},\;a_{2M+i},\;a_{2M+j},\;a_{3M+i},\;a_{3M+j},\;\ldots\;.
\end{equation*}
This sequence of positive numbers terminates as soon as either $kM+i$ or $kM+j$ becomes greater than $n$ for the first time. The entry $a_{kM+i}$ is the last nonzero element if and only if
\[
kM+i\leq n<kM+j\Longleftrightarrow \frac{n-j}{M}<k\leq\frac{n-i}{M}\Longleftrightarrow k=\bigg\lfloor\frac{n-i}{M}\bigg\rfloor=\bigg\lceil\frac{n-j}{M}\bigg\rceil,
\]
since the difference between the bounds for $k$ is strictly smaller than $1$. The entry $a_{kM+j}$ is the last nonzero element if and only if $kM+j\leq n<(k+1)M+i$ or, equivalently,
\[
k \leq \frac{n-j}{M}<\frac{n-i}{M}<k+1\Longleftrightarrow k=\bigg\lfloor\frac{n-i}{M}\bigg\rfloor=\bigg\lfloor\frac{n-j}{M}\bigg\rfloor.
\]
There are $2\lfloor(n-i)/M\rfloor$ nonzero leading principal minors of $H_M^{(ij)}$ in the first case and $2\lfloor(n-i)/M\rfloor+1$ nonzero leading principal minors in the second. By Routh-Hurwitz Theorem~\ref{HurwitzTheorem}, the number of nonzero leading principal minors equals the degree $m$ of the polynomial $P^{(ij)}(x)$. \hfill
\qed
\end{proof}

\bigskip

For our next (somewhat technical) lemma, we need to define $\alpha$ and $\beta$ by
\begin{equation}\label{svkalpha}
\alpha=\frac{n-i}{M}-\bigg\lfloor\frac{n-i}{M}\bigg\rfloor,\quad \beta=\frac{n-j}{M}-\bigg\lfloor\frac{n-j}{M}\bigg\rfloor.
\end{equation}

\begin{lemma}\label{LemPtoF2} Let $\varepsilon(m)=\frac{1-(-1)^m}{2}$. Then
\[
P^{(ij)}_0\left(x\right)=x^{\varepsilon(m)-2\alpha}f_i\left(x^{\frac{2}{M}}\right),\quad P^{(ij)}_1\left(x\right)=x\cdot x^{-\varepsilon(m)-2\beta}f_j\left(x^{\frac{2}{M}}\right).
\]
\end{lemma}
\begin{proof} By Lemma \ref{LemPtoF},
\[
P^{(ij)}_0(x)=x^{m-2\frac{n-i}{M}}f_i\left(x^{\frac{2}{M}}\right),\quad P^{(ij)}_1(x)=x^{m-1-2\frac{n-i}{M}}f_i\left(x^{\frac{2}{M}}\right).
\]
By Lemma \ref{degreeformula},
\[
m-2\frac{n-i}{M}=\begin{cases}-2\alpha\;&\text{ if }\; m\;\text{is even}\\
-2\alpha+1\;&\text{ if }\; m\;\text{is odd,}\end{cases}
\]
which proves the first formula. To prove the second one, we observe that
\[
m-2\frac{n-j}{M}-1=\begin{cases}2\left(\Big\lfloor\frac{n-i}{M}\Big\rfloor-\frac{n-j}{M}\right)-1\;&\text{ if }\; m\;\text{is even}\\[2mm]
2\left(\Big\lfloor\frac{n-i}{M}\Big\rfloor-\frac{n-j}{M}\right)\;&\text{ if }\; m\;\text{is odd.}
\end{cases}
\]
By Lemma \ref{degreeformula},
\[
2\left(\Big\lfloor\frac{n-i}{M}\Big\rfloor-\frac{n-j}{M}\right)=2\left(\Big\lfloor\frac{n-j}{M}\Big\rfloor-\frac{n-j}{M}\right)=-2\beta.
\]
if $m$ is odd. For even $m$ we have
\[
\beta=\left(\bigg\lfloor\frac{n-i}{M}\bigg\rfloor-\frac{n-j}{M}\right)=\left(\bigg\lceil\frac{n-j}{M}\bigg\rceil-\frac{n-j}{M}\right)\in (0,1).
\]
It follows that $\dfrac{n-j}{M}$ cannot be an integer and therefore
\[
\bigg\lceil\frac{n-j}{M}\bigg\rceil=\bigg\lfloor\frac{n-j}{M}\bigg\rfloor+1,
\]
implying the second formula.
\qed
\end{proof}

\begin{corollary}
Let $f(x) = f_0(x)+f_1(x)+ \cdots + f_{M-1}(x)$ and let all the coefficients $h_0, \ \ldots, \ h_n$ in the generalized Euclidean algorithm with step $M$ be positive. Then
\begin{itemize}
\item[$(a)$]\;every polynomial
\[
P^{(ij)}(x) =P^{(ij)}_0(x)+P^{(ij)}_1(x)=x^{\varepsilon(m)-2\alpha}f_i\left(x^{{2}/{M}}\right)+x\cdot x^{-\varepsilon(m)-2\beta}f_j\left(x^{{2}/ {M}}\right)
\]
is stable;
\item[$(b)$] all coefficients $h^{ij}_0,\,h^{ij}_1 \ldots, \, h^{ij}_k$ in the (ordinary) Euclidean algorithm applied to the pair of polynomials $(f_i, f_j)$ are positive.
    \end{itemize}
\end{corollary}
\begin{proof} Item (a) follows from Lemma \ref{LemPtoF2} and the classical Routh-Hurwitz 
Theorem~\ref{HurwitzTheorem}.

To prove (b) we apply Lemma \ref{LemPtoF} with $x$ replaced by $\sqrt{x}$, $x>0$. We have
\[
R_{ij}(x)=\frac{f_i(x)}{f_j(x)}=x^{(j-i)-M/2}\cdot\frac{P^{(ij)}_0(x^{M/2})}{P^{(ij)}_1(x^{M/2})}.
\]
It is well known (see \cite[Corollary 7.33]{JT}) that the $C$-fraction expansion of the rational function
\[
\frac{P^{(ij)}_0(z)}{P^{(ij)}_1(z)}=d_1z+\dfrac{1}{d_2z {+} \dfrac{1}{d_3z + \dfrac{1}{\ddots+\dfrac{1}{d_sz}}}}
\]
has all positive coefficients $d_1$, $\ldots$, $d_s$ whenever $P(z)$ is a stable polynomial. It follows that
\begin{eqnarray*}
R_{ij}(x)&=&\frac{f_i(x)}{f_j(x)}=x^{(j-i)-M/2}\cdot\frac{P^{(ij)}_0(x^{M/2})}{P^{(ij)}_1(x^{M/2})}=
x^{(j-i)-M/2}\cdot
\Big\{d_1x^{M/2}+\dfrac{1}{d_2x^{M/2} + \dfrac{1}{d_3x^{M/2} +\cdots }} \Big\} \\
&=& d_1x^{(j-i)}+\dfrac{x^{(j-i)-M/2}}{d_2x^{M/2} + \dfrac{1}{d_3x^{M/2} +\cdots}}\,=
d_1x^{(j-i)}+\dfrac{1}{d_2x^{M-(j-i)} + \dfrac{x^{M/2-(j-i)}}{d_3x^{M/2} +\cdots}}\,
\\ &=& d_1x^{(j-i)}+\dfrac{1}{d_2x^{M-(j-i)} + \dfrac{1}{d_3x^{(j-i)} +\cdots}}\,.
\end{eqnarray*}
Comparing this continued fraction with the continued fraction in \eqref{ku5} and invoking the uniqueness theorem for $C$-fractions (see \cite{JT}), we obtain that
\[
0<d_l=\dfrac{h^{ij}_{\ell-1}}{h^{ij}_\ell}\Rightarrow h^{ij}_\ell>0,
\]
since $h^{ij}_0=a_i>0$ and $h^{ij}_1=a_j>0$.\qed
\end{proof}


\section{Main Results}\label{sec:proof}

This section is devoted to the proofs of  Theorem~\ref{maintheorem} for $M\geq 2$ and of our main Theorem~\ref{coolthm}. Note that these theorems differ both in assumptions and in conclusions.
In particular, the statement of Theorem~\ref{coolthm} includes the important case $M=1$, which has to be excluded in the formulation of Theorem~\ref{maintheorem}. We hope this does not unduly confuse the reader.  We shall begin with Theorem~\ref{maintheorem} and shall build the proof lemma by lemma, culminating in its proof. We shall then prove Theorem~\ref{coolthm} using a different method.

Given $M \geq 2$ nonzero numbers $w,\,z_1, \ \ldots, \ z_{M-1}$ in the cone 
$K_{0,\pi/M}=\left\{u: 0\leq \arg(u)\leq\dfrac{\pi}{M}\right\} $
and positive numbers $a_n>0$, $n\geq M$, we define recursively
\begin{equation}\label{svk11}z_n\eqbd a_nw + \frac{1}{z_{n-1}\cdot \ldots\cdot z_{n - (M-1)}}.
\end{equation}
For $n\geq M$ we denote
\[ {S}(n)\eqbd \{n-(M-2),\ldots,n\}
\]
and define the partition $S(n)={S}_{-}(n)\cup{S}_{+}(n)$ of this set by
\begin{equation*}\label{svk12}\begin{aligned}
{S}_{+}(n)&=\{i\in {S}(n):\arg(z_i)\geq 0\},\\ {S}_{-}(n)&=\{i\in {S}(n):\arg(z_i)<0\}.\end{aligned}
\end{equation*}

\begin{lemma}\label{lem67} Let $\{z_1, \ \ldots, \ z_n\}$ be a sequence defined by \eqref{svk11}. Suppose that $0\leq\arg(w)\leq \frac{\pi}{M}$ and
\begin{equation}\label{svk14}
- \pi+\frac{\pi}{M} \leq \arg(z_i) \leq \frac{\pi}{M},\quad i=M,\ldots,n.
\end{equation}
Then
\begin{equation*}\label{svk15} \max_{S={S}_{+}(n),\,{S}_{-}(n)}\Big|\sum_{i\in S}\arg (z_i)\Big| \leq \frac{\pi(M-1)}{M} < \pi.
\end{equation*}
\end{lemma}

\begin{proof} We use induction on $n$. Let $n = M$. Then ${S}(M) = \{2, \ldots, M\}$. Since $\arg(z_i)\geq 0$ for $i=2,\ldots,M-1$, we have either ${S}_{+}(M) = {S}(M)$ or ${S}_{+}(M) = {S}(M) \setminus\{M\}$, hence ${S}_{-}(M) = \emptyset$ or ${S}_{-}(M) =\{M\}$. Suppose first that ${S}_{-}(M) = \emptyset$. Since we assume that $-\pi<\arg(z_M)\leq \pi/M$, we obtain
\begin{equation*}
0\leq \arg(z_M)\leq \frac{\pi}{M},
\end{equation*}
which implies that
\begin{equation*}
\sum_{i\in {S}_{+}(M)}\arg (z_i) = \sum_{i=2}^{M}\arg (z_i) \leq  \dfrac{\pi(M-1)}{M} < \pi.
\end{equation*}
Let us consider the case  ${S}_{-}(M) = \{M\}$. Then ${S}_{+}(M)=\{2,\ldots,M-1\}$,
implying that
\begin{equation*}\label{svk20}
0\leq\sum_{i\in {S}_{+}(M)}\arg (z_i) = \sum_{i=2}^{M-1}\arg (z_i) \leq \frac{(M-2)\pi}{M}<\frac{(M-1)\pi}{M}.
\end{equation*}
For ${S}_{-}(M) = \{M\}$ we set
\begin{equation*}\label{svk22}
u=\frac{1}{z_{M-1}\cdot \ldots\cdot z_1}.
\end{equation*}
Then
\[
-\frac{(M-1)\pi}{M}\leq -\sum_{i=1}^{M-1}\arg(z_1)=\arg(u)\leq 0.
\]
Since ${S}_{-}(M) = \{M\}$, we obtain
\begin{equation*}\label{svk21}
\arg(a_Mw + u)=\arg(z_M)\leq 0.
\end{equation*}
Both $w$ and $u$ are in the half-plane
\[
\left\{z:-\frac{(M-1)\pi}{M}\leq\arg(z)\leq\frac{\pi}{M}\right\}.
\]
Therefore the sum $a_Mw+u$ is in the same half-plane too. It follows that
\[
-\frac{(M-1)\pi}{M}\leq\arg(a_Mw+u)=\arg(z_M)\leq 0.
\]
This proves the Lemma for $n=M$.

Suppose that the Lemma is proved for $n$. Let us prove it for $n+1$.  We have
\begin{equation*}
{S}(n) = \{n - (M-2), \ \ldots, \ n\},\quad {S}(n+1) = \{n - (M-3), \ \ldots, \ n+1\}.
\end{equation*}
If
\begin{equation*}
\max_{S={S}_{+}(n+1),\,{S}_{-}(n+1)}\Big|\sum_{i\in S}\arg (z_i)\Big| = \sum_{i\in {S}_{+}(n+1)}\arg (z_i),
\end{equation*}
then this maximum cannot exceed $\frac{\pi(M-1)}{M}$, since the sum contains not more then $M-1$ summands each of which is nonnegative and does not exceed $\pi/M$ by our assumption that $\arg(z_i)\leq \pi/M$.

Now let
 \begin{equation*}\max_{S={S}_{+}(n+1),\,{S}_{-}(n+1)}\Big|\sum_{i\in S}\arg (z_i)\Big| = \sum_{i\in {S}_{-}(n+1)}-\arg (z_i).
 \end{equation*}
Suppose first that $n+1 \notin {S}_{-}(n+1)$, i.e., $\arg(z_{n+1}) \geq 0$. Then ${S}_{-}(n+1)\subset {S}_{-}(n)$, implying that
\begin{equation*}\Big|\sum_{i\in{S}_{-}(n+1)}\arg (z_i)\Big| \leq  \Big|\sum_{i\in {S}_{-}(n)}\arg (z_i)\Big| \leq \dfrac{\pi(M-1)}{M}
\end{equation*}
by the induction hypothesis.

Now consider the case when $n+1 \in {S}_{-}(n+1)$, i.e., $\arg(z_{n+1}) < 0$.
We have
\begin{equation*}
\arg(z_{n+1})= \arg\left(a_{n+1}w + \frac{1}{z_{n}\cdot \ldots\cdot z_{n - (M-2)}}\right) < 0.
\end{equation*}
Since $\arg(a_{n+1}w)\geq 0$, we see that
\[0>\arg(z_{n+1})>\arg\left(\frac{1}{z_{n}\cdot \ldots\cdot z_{n - (M-2)}}\right) = - \sum_{i= n - (M-2)}^{n}\arg (z_i)=\sum_{i\in S_{-}(n)}\left\{-\arg(z_i)\right\}-\sum_{i\in S_{+}(n)}\arg(z_i).\\
\]
If $n-(M-2)\in S_{+}(n)$, then
\begin{equation*}
0>\sum_{i\in{S}_{-}(n+1)}\arg (z_i)=\arg(z_{n+1})+\sum_{i\in S_{-}(n)}\arg(z_i)>-\sum_{i\in S_{+}(n)}\arg(z_i)>-\frac{\pi(M-1)}{M}
\end{equation*}
by the induction hypothesis.
If $n-(M-2)\in S_{-}(n)$, then
\begin{eqnarray*}
\qquad \qquad 0 &>&\sum_{i\in{S}_{-}(n+1)}\arg (z_i)=\arg(z_{n+1})+\sum_{i\in S_{-}(n)}\arg(z_i)-\arg(z_{n-(M-2)}) \\ &>& -\arg(z_{n-(M-2)})-\sum_{i\in S_{+}(n)}\arg(z_i)>-\frac{\pi(M-1)}{M}.\qquad \qquad \qquad \qquad \qquad   \qquad  \quad \qquad \qed
\end{eqnarray*}
\end{proof}

Now suppose the generalized Euclidean algorithm produces all $h_j$ positive. Then
 the sequence
\begin{equation*}
z_i = \frac{f_{n-i}(z)}{f_{n-i+1}(z)}, \quad i = 1, \, \ldots,\, n,
\end{equation*}
satisfies the conditions of Lemma \ref{lem67}, where $0\leq \arg(z)\leq \pi/M$, $w=z$, $a_i=h_{n-i}/h_{n-i+1}$. Indeed, the inequalities \eqref{svk14} hold by Corollary \ref{cor22}. Each arithmetic polynomial $f_{n-(M-1)+j}$, $j=0,\ldots, M-1$, has difference $M$ and degree $(M-1)-j$. Hence  these polynomials are monomials and therefore all points
\[
z_{(M-1)-j}=\frac{f_{n-(M-1)+j}(z)}{f_{n-(M-1)+j+1}(z)}=\frac{h_{n-(M-1)+j}}{h_{n-(M-1)+j+1}}z
\]
for $j=0,1,\ldots, M-2$ lie in the cone $K_{0,\pi/M}$. If $M\leq i\leq n$ then, by the generalized Euclidean algorithm and Corollary \ref{corGEA},
\begin{eqnarray*}
z_{i}&=&\frac{f_{n-i}}{f_{n-i+1}(z)}=\frac{h_{n-i}}{h_{n-i+1}}z+\frac{f_{n-i+M}(z)}{f_{n-i+1}(z)}=a_i z+\frac{1}{\frac{f_{n-i+1}(z)}{f_{n-i-M}(z)}} \\ &=&
a_i z+\frac{1}{\frac{f_{n-i+1}(z)}{f_{n-i+2}(z)}\cdot\frac{f_{n-i+2}(z)}{f_{n-i+3}(z)}\cdots\frac{f_{n-i+(M-1)}(z)}{f_{n-i+M}(z)}}=
a_i z+\frac{1}{z_{i-1}\cdot z_{i-2}\cdots z_{i-(M-1)}},
\end{eqnarray*}
which implies \eqref{svk11}. Now using Lemma \ref{lem67}, we obtain the following Corollary.
\begin{corollary}\label{cor34} If $z\in K_{0,\pi/M}$ and $n\geq M-1$, then
\begin{equation}\label{svk105} \max_{S={S}_{+}(n),\,{S}_{-}(n)}\Big|\sum_{i\in S}\arg \left(\frac{f_{n-i}(z)}{f_{n-i+1}(z)}\right)\Big| \leq \frac{\pi(M-1)}{M} < \pi.
\end{equation}
\end{corollary} \medskip

\noindent
The mapping $i\longmapsto n-i$ maps the set of integers $S(n)=\{n-(M-2),\ldots,n\}$ bijectively onto the set $S(M-2)=\{0,\ldots,M-2\}$. This transforms \eqref{svk105} into its 'dual' form
\begin{equation}\label{svk205} \max_{S={S}_{+}(M-2),\,{S}_{-}(M-2)}\Big|\sum_{i\in S}\arg \left(\frac{f_{i}(z)}{f_{i+1}(z)}\right)\Big| \leq \frac{\pi(M-1)}{M} < \pi.
\end{equation}
\medskip

Now we can state and prove our main result.
\begin{theorem} \label{maintheorem} Let $f(x) = a_0x^n + a_1x^{n-1} + \cdots + a_n$ $(a_0, \ a_1, \ \ldots, \ a_n \in {\mathbb R})$ be a polynomial of degree $n$ and $M$ be a positive integer $(2 \leq M \leq n)$. Let all the leading coefficients $h_0, \ \ldots, \ h_n$ of the polynomials $f_0, \ \ldots, \ f_n$ obtained by applying the generalized Euclidean algorithm with step $M$, be positive. Then $f(z) = 0$ implies $|\arg(z)| > \frac{\pi}{M}$.
\end{theorem}

 \begin{proof} Since all coefficients of $f$ are real, its complex roots occur in complex conjugate pairs. Hence it is sufficient to prove that $f(z)$ has no roots in the sector $0 \leq \arg(z) \leq \frac{\pi}{M}$.
The equation $f(z) = 0$ is equivalent to
\begin{equation*}\label{seq2}
f_0(z) + f_1(z) + \cdots + f_{M-1}(z) = 0,
\end{equation*}
or to the equation
\begin{equation}\label{svk010}\frac{f_0}{f_{M-1}}(z) + \frac{f_1}{f_{M-1}}(z) + \cdots + \frac{f_{M-2}}{f_{M-1}}(z) = -1.
\end{equation}
We represent each term in the above sum as a product
\[
\frac{f_k}{f_{M-1}}=\frac{f_k}{f_{k+1}}\cdot\frac{f_{k+1}}{f_{k+2}}\cdots\frac{f_{M-2}}{f_{M-1}}
\]
and apply Horner's algorithm so that the equation \eqref{svk010} takes the form:
\begin{equation}\label{ku09}
\frac{f_{M-2}}{f_{M-1}}(z)\left(\frac{f_{M-3}}{f_{M-2}}(z) \cdots\left(\frac{f_1}{f_2}(z)\left(\frac{f_0}{f_1}(z)+1\right)+1\right) \cdots +1\right) = -1.
\end{equation}
The proof of the theorem will be completed if we can show that the argument of the left-hand side 
of~\eqref{ku09} is smaller than $\pi$ in absolute value. As a matter of fact, we will show that it does not exceed $\pi(M-1)/M$ in absolute value. To do this we rewrite the left-hand side of~\eqref{ku09} in the form of the following recurrence relation
\begin{equation*}
u_{k+1}=y_{k+1}(u_k+1),\quad y_k=\frac{f_{k-1}}{f_k}(z),\quad k=1,\ldots,M-1,\quad  u_1=y_1.
\end{equation*}
As before, for every $k$ we define the decomposition
\[
\{1,2,\ldots,k\}=T_{+}(k)\cup T_{-}(k),
\]
where $j\in T_{-}(k)$ if and only if $\arg(y_j)<0$ and $j\in T_{+}(k)$ if and only if $\arg(y_j)\geq0$.
\begin{lemma}\label{CDC} For every $k=1,\ldots, M-1$,
\begin{equation}\label{svk407}
|\arg(u_k)|\leq \max_{T=T_{+}(k),\,T_{-}(k)}\Big|\sum_{i\in T}\arg (y_i)\Big|.
\end{equation}
Moreover, if $\arg(u_k) \geq 0$, then
\begin{equation}\label{svk408}\arg(u_k) \leq \sum_{i\in T_{+}(k)}\arg (y_i) \leq \max_{T=T_{+}(k),\,T_{-}(k)}\Big|\sum_{i\in T}\arg (y_i)\Big|;
\end{equation}
if $\arg(u_k)<0$, then
\begin{equation}\label{svk409}
-\arg(u_k) \leq -\sum_{i\in T_{-}(k)}\arg (y_i) \leq \max_{T=T_{+}(k),\,T_{-}(k)}\Big|\sum_{i\in T}\arg (y_i)\Big|.
\end{equation}
\end{lemma}
\begin{proof} We prove the lemma by induction. Since $u_1=y_1(0+1)$ it is clear the statement holds for $k=1$. Suppose that the lemma is proved for $k$. Then by \eqref{svk205} $\{y_1,y_2,\ldots, y_k, y_{k+1}\}$ is a sequence in $K_{-\pi+\frac{\pi}{M},\frac{\pi}{M}}$ satisfying
\[
\max_{T=T_{+}(k+1),\,T_{-}(k+1)}\Big|\sum_{i\in T}\arg (y_i)\Big|\leq\frac{\pi(M-1)}{M}.
\]
We have
\[
\arg(u_{k+1})=\arg(y_{k+1})+\arg(u_k + 1).
\]
 Suppose first that $\arg(u_{k+1})\geq 0$. Then at least one of $\arg(y_{k+1})$ and $\arg(u_k + 1)$ is non-negative.

If $\arg(y_{k+1})\geq 0$ and $\arg(u_k + 1)\leq 0$, then
\[
\arg(u_{k+1})\leq \arg(y_{k+1})\leq\arg(y_{k+1})+ \sum_{i\in T_{+}(k)}\arg (y_i)=\sum_{i\in T_{+}(k+1)}\arg (y_i).
\]
If $\arg(y_{k+1})\geq 0$ and $\arg(u_k + 1)> 0$, then $\arg(u_k)>0$ and
\begin{eqnarray*}
\arg(u_{k+1})&=&\arg(y_{k+1})+\arg(u_k + 1)\leq\arg(y_{k+1})+\arg(u_k) \\ &\leq& \arg(y_{k+1})+\sum_{i\in T_{+}(k)}\arg (y_i)=\sum_{i\in T_{+}(k+1)}\arg (y_i).
\end{eqnarray*}
If $\arg(y_{k+1})<0$ and $\arg(u_k + 1)> 0$, then
\begin{eqnarray*}
\arg(u_{k+1}) &=&\arg(y_{k+1})+\arg(u_k + 1)\leq\arg(y_{k+1})+\arg(u_k) \\ &\leq&  \arg(u_k)\leq \sum_{i\in T_{+}(k)}\arg (z_i)=\sum_{i\in T_{+}(k+1)}\arg (y_i).
\end{eqnarray*}

\noindent
Suppose now that $\arg(u_{k+1})<0$. Then at least one of $\arg(y_{k+1})$ and $\arg(u_k + 1)$ is negative.
If $\arg(y_{k+1})\geq 0$ and $\arg(u_k + 1)< 0$, then
\[
0>\arg(u_{k+1})\geq \arg(u_{k}+1)\geq\arg(u_k)\geq \sum_{i\in T_{-}(k)}\arg (y_i)=\sum_{i\in T_{-}(k+1)}\arg (y_i).
\]
If $\arg(y_{k+1})< 0$ and $\arg(u_k + 1)\geq 0$, then
\[
0>\arg(u_{k+1})=\arg(y_{k+1})+\arg(u_k + 1)\geq \arg(y_{k+1})\geq\sum_{i\in T_{-}(k+1)}\arg (y_i).
\]
If $\arg(y_{k+1})<0$ and $\arg(u_k + 1)<0$, then $\arg(u_k+1)>\arg(u_k)$ and
\begin{eqnarray*}
\qquad \qquad \qquad 0 >\arg(u_{k+1}) &=&\arg(y_{k+1})+\arg(u_k + 1)\geq \arg(y_{k+1})+\arg(u_k) \\ &\geq & \arg(y_{k+1}) +\sum_{i\in T_{-}(k)}\arg (y_i)=\sum_{i\in T_{-}(k+1)}\arg (y_i).  \qquad \qquad \qquad \qquad \qed
\end{eqnarray*}

\end{proof} \noindent 
Returning to the proof of the main theorem,  we observe that $u_{M-1}$ is the left-hand side of \eqref{ku09}. By Lemma \ref{CDC} and by the inequality \eqref{svk205} we conclude that $u_{M-1}$
cannot equal $-1$.\qed
\end{proof}
\medskip

Let us illustrate this theorem for $n=5$ and $M=3$.

\begin{corollary}\label{SVKCor1} Let $a_0$, $a_1$, $a_2$ be positive as well as the leading coefficients $h_3$, $h_4$, $h_5$ of the polynomials constructed by the generalized Euclidean algorithm with step $M=3$  for the polynomial
\[
f(x)=a_0x^5+a_1x^4+a_2x^3+a_3x^2+a_4x+a_5.
\]
Suppose also that
\begin{equation}\label{6745}
\frac{a_3}{a_0}>\frac{a_4}{a_1}>\frac{a_5}{a_2}.
\end{equation}
Then the polynomial $f(z)$ does not vanish in the closed sector $$\big\{z:|\arg(z)|\leq\frac{\pi}{3}\big\}.$$
\end{corollary}
\begin{proof} First we determine $f_0$, $f_1$, $f_2$:
\[
\begin{aligned}f_0(x)&=a_0x^5+a_3x^2;\\ f_1(x)&=a_1x^4+a_4x;\\f_2(x)&=a_2x^3+a_5.\end{aligned}
\]
Next,
\[
\begin{aligned}f_0(x)&=\frac{a_0}{a_1}x\cdot f_1(x)+\left(a_3-\frac{a_0a_4}{a_1}\right)x^2\\ f_1(x)&=\frac{a_1}{a_2}x\cdot f_2(x)+\left(a_4-\frac{a_1a_5}{a_2}\right)x\\f_2(x)&=\frac{a_1a_2}{a_1a_3-a_0a_4}xf_3+a_5\end{aligned}\quad\Longrightarrow\quad
\begin{aligned}f_3(x)&= \left(a_3-\frac{a_0a_4}{a_1}\right)x^2\\f_4(x)&=\left(a_4-\frac{a_1a_5}{a_2}\right)x\\f_5(x)&=a_5\end{aligned}
\]
Then, by \eqref{6745},
\[
h_0=a_0>0,\;h_1=a_1>0,\;h_2=a_2>0,\; h_5=a_5>0, h_3=a_0\left(\frac{a_3}{a_0}-\frac{a_4}{a_1}\right)>0,\;h_4=a_1\left(\frac{a_4}{a_1}-\frac{a_5}{a_2}\right)>0.
\]
\end{proof}

Notice that \eqref{6745} implies that $a_3$, $a_4$ are positive as soon as $a_0$, $a_1$, $a_2$, $a_5$ are positive.

 Using MATLAB or a mere scientific calculator, we can verify that the polynomial
\[
f(x)=x^5+x^4+x^3+1.001x^2+x+0.999,
\]
which satisfies the conditions of Corollary \ref{SVKCor1}, has the following roots:
\[
x_1=-1,\; x_{2,3}=-0.49975\pm 0.865592 i,\; x_{4,5}=0.49975\pm 0.86617 i.
\]

The roots $x_1$, $x_2$, and $x_3$ are located in the left half-plane. The slopes of the vectors $x_4$ and $x_5$ equal
\[
\pm \frac{0.86617}{0.49975}=\pm 1.73321.
\]
The slopes of the boundaries of the sector with $M=3$ equal
\[
\pm\tan \left(\frac{\pi}{3}\right)=\pm\sqrt{3}=\pm 1.73205.
\]
This illustrates numerically the conclusion of Corollary \ref{SVKCor1} that the polynomial $f(z)$ does not vanish in the sector plotted below:

 
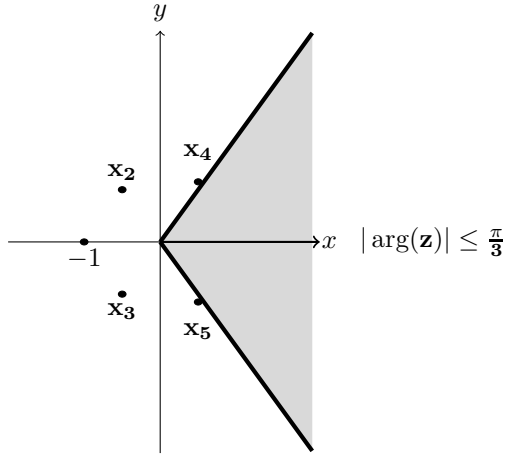
\begin{center}
\begin{tikzpicture}[xscale=1,yscale=0.8]
\draw[->] (-2,0) -- (2,0);
\draw[->] (0,-3.5) -- (0,3.5);

\path[fill=gray!30] (0,0) -- (2,3.4641) -- (2,-3.4641) -- (0,0);
\draw[black, ultra thick, domain=0:2] plot (\x, {-1.73205*\x});
\draw[->,black,thick] (0,0) -- (2.1,0);
\draw[black, ultra thick, domain=0:2] plot (\x, {1.73205*\x});
\draw[fill=black] (-1,0) circle [radius=0.05];
\draw[fill=black] (-0.49975,0.86559) circle [radius=0.05];
\draw[fill=black] (-0.49975,-0.86559) circle [radius=0.05];
\draw[fill=black] (0.49975,0.99617) circle [radius=0.05];
\draw[fill=black] (0.49975,-0.99617) circle [radius=0.05];

\foreach \x in {-1}
\draw (\x cm,1pt) -- (\x cm,1pt) node[anchor=north] {$\x$};
\node[anchor=north] at (-0.49975,-0.86559) {$\mathbf{x_3}$};
\node[anchor=south] at (-0.49975,0.86559) {$\mathbf{x_2}$};
\node[anchor=south] at (0.49975,1.2) {$\mathbf{x_4}$};
\node[anchor=north] at (0.49975,-1.2) {$\mathbf{x_5}$};

\node[anchor=west] at (2.5,0) {$\mathbf{|\arg(z)|\leq\frac{\pi}{3}}$};
\node[anchor=south] at (0,3.5) {$y$};
\node[anchor=west] at (2,0) {$x$};
\end{tikzpicture}
\end{center}

\noindent
Notice that inequalities \eqref{6745} are of course equivalent to inequalities for the following minors of $H_3$:
\[
\begin{vmatrix}a_1&a_4\\a_0&a_3\end{vmatrix}>0, \qquad \begin{vmatrix}a_2&a_5\\a_1&a_4\end{vmatrix}>0.
\]

\smallskip
\noindent
Theorem~\ref{coolthm} can be proved in various ways. The idea of the proof below was suggested  by Mikhail Tyaglov. \smallskip

\proofof{Theorem~\ref{coolthm}} The proof will proceed by deriving a contradiction from the assumptions $f(z)=0$, $\arg z\in (0,{\pi\over M})$. Since $f$ is a real polynomial, this will also establish by conjugation that the assumptions $f(z)=0$, $\arg z\in (-{\pi\over M}, 0)$ also lead to a contradiction. Finally, $f$ cannot have positive zeros since all its coefficients are nonnegative, which will thus show that $f$ has no zeros in the whole sector $|\arg(z)|< {\pi \over M}$.
Disclaimer: this result of course does not prevent $f$ from having zero at the origin!

Before we embark on the proof, let us make another reduction. Without loss of generality we can assume that $f_0(z)$ is a polynomial in $z^M$ by multiplying $f(z)$ by a suitable power of $z$.  Of course, this may create a (multiple) zero at the origin, but as we already remarked, we are not concerned with those. Note that, once $f_0(z)\bdeq p_0(z^M)$ is a polynomial in $z^M$, so are $zf_1(z) \bdeq p_1(z^M)$, $z^2f_2(z)\bdeq p_2(z^M)$, $\ldots$, $z^{M-1}f_{M-1}(z)\bdeq p_{M-1}(z^M)$. Some of them may be identically zero though.

Now, let $\arg z \in (0,{\pi\over M})$ and suppose
$$0=f(z)=f_0(z)+\cdots+f_{M-1}(z)=p_0(z^M)+{p_1(z^M)\over z}+\cdots + 
{p_{M-1}(z^M)\over z^{M-1}}.$$
Some of the $f_j$, $j>0$ may well be identically zero, those can be simply ignored.

The Hurwitz matrix for each pair $(p_i,p_j)$, $i<j$, is a submatrix of the generalized Hurwitz matrix $H_M(f)$ and is therefore totally nonnegative. Hence each resulting nonzero polynomial $p_j$ must have only nonpositive real roots by, e.g., \cite[Theorem~3.44(1)]{HT} and $p_j/p_0$ is an $R$-function of positive type, i.e., a function mapping the open upper half-plane $\C_+$ to itself. A small but crucial technicality: \cite[Theorem~3.44(1)]{HT} talks about $R$-functions of positive type and lists the coefficients of the pair in the order opposite to ours here.  Since $z^M$ does not lie on the negative real half-line for $\arg z \in (0,\pi/M)$, neither $p_j(z^M)$ nor $f_j(z)$ can vanish.
Furthermore,  $\arg z\in (0,\pi/M)$, so $z^M$ lies in $\C_+$, hence
$$ \arg f_j(z)= \arg \big(  {p_j(z^M)\over z^j p_0(z^M) } \big) = -j\arg z +\phi   \quad  \text{\rm for some}\quad \phi\in (0, \pi). $$
Hence $\arg f_j(z) \in (-j\pi/M,(M-j)\pi/M)$ and $f_j(z)$ cannot take on real negative values either.

As $p_0$ is certainly nonzero due to its leading term, we can divide everything through by it and get
\begin{equation}\label{cute}
 {p_1(z^M) \over z p_0(z^M)}+ {p_2(z^M)\over z^2 p_0(z^M)} +  \cdots + 
 {p_{M-1}(z^M)\over z^{M-1} p_0(z^M)}  = -1.   
\end{equation}

Among the nonzero terms, some lie in the half-open upper half-plane $\C^*_+\eqbd \{ w: \arg w \in [0,\pi)\} $ and some in the open lower half-plane $\C_-\eqbd \{ w: \arg w \in (-\pi,0)\} $. No terms lie on the negative real axis, as we already observed.

Moreover, the sum of all the terms in $\C^*_+$ is a vector in $\C^*_+$ as well that lies between the positive real half-line and the term with the \emph{largest} argument in the upper half-plane, say, $z^{M-i} p_i(z^M)/p_0(z^M)$. Likewise, the sum of all the terms in $\C_-$ is a vector in $\C_-$ that lies between the positive real half-line and the terms with the smallest (negative) argument, say, $z^{M-j} p_j(z^M)/p_0(z^M)$; $j\neq i$. This implies that 
\begin{equation}\label{conditionpi}
\arg \left(  {p_i(z^M)\over z^i p_0(z^M)} \right) -
\arg \left(  { p_j(z^M)\over z^j p_0(z^M) }\right) >\pi.
\end{equation}
 Indeed, if not, then \emph{all} vectors in~\eqref{cute} lie in the half-plane to the right of the vector
$z^{M-i} \, {p_i(z^M)\over p_0(z^M)} $, and hence their sum cannot equal $-1$.

Again, ${p_i / p_0}$, $p_j/p_0$ and $p_j/p_i$ (if $i<j$) are $R$-functions of positive type by~\cite[Theorem~3.44(1)]{HT}, so
\begin{eqnarray*} \arg \big(  {p_i(z^M)\over z^i p_0(z^M) } \big) = -i\arg z +\phi  & \quad & \text{\rm for some}\quad \phi\in (0, \pi). \\
\arg \big(  {p_j(z^M)\over z^j p_0(z^M)} \big) = -j\arg z +\psi & \quad & \text{\rm for some}\quad \psi\in (0,\pi),
\end{eqnarray*}
and $\arg {p_i(z^M)\over p_j(z^M)}=\phi - \psi$, $\arg {p_j(z^M)\over p_i(z^M)}=\psi - \phi$
since $\phi-\psi$ lies within the correct interval $(-\pi,\pi)$.
\begin{eqnarray*}
& {\rm If} \; i<j, \; {\rm then}  \; \phi-\psi \in (-\pi,0), \; {\rm hence} \; 
(j-i)\arg z +\phi - \psi \in \left(-\pi, (j-i){\pi\over M} \right).  \qquad \qquad \quad  \qquad \qquad \\
& {\rm If}\; j<i, \; {\rm then} \; \phi-\psi\in (0, \pi), \;  {\rm hence} \;  (j-i)\arg z +\phi - \psi \in \left(-(i-j){\pi\over M}, \pi-(i-j){\pi\over M} \right). \qquad\qquad\;\;
\end{eqnarray*}

In both cases $(j-i)\arg z +\phi - \psi$ cannot exceed $\pi$. But the inequality~\eqref{conditionpi} 
 means $$ (j-i)\arg z +\phi - \psi >\pi.$$ 
Contradiction! And Theorem~\ref{coolthm} is proven. \eop


\section{Factorization of the generalized Hurwitz matrix}\label{sec:factorization}

The ordinary Hurwitz matrix turns out to admit a particularly simple factorization~\cite{HOL}, with factors determined by the leading coefficients arising in the (ordinary) Euclidean algorithm applied to the even and odd parts of a polynomial. Our natural question now is whether an analogous factorization holds for generalized Hurwitz matrices. This turns out to be true!

Before we embark on a proof, we will switch from the matrix $H_M$ to its counterpart 
$$\widetilde{H}_M\eqbd \widetilde{H}_M(f) \eqbd 
\begin{pmatrix} a_0 & a_M & a_{2M} & \cdots \\
0 & a_{M-1} & a_{2M-1} & \cdots \\
0 & a_{M-2} & a_{2M-2} & \cdots \\
\vdots & \vdots & \vdots & \ddots \\
0 & a_0 & a_M & \cdots \\
0 & 0 & a_{M-1} & \cdots \\
\vdots & \vdots & \vdots & \ddots \\
\end{pmatrix}
$$
This is in keeping with the notation of~\cite{HOL}; it also makes for more elegant formulas. The reader intent on factoring the original matrix $H_M$ may do so by following the ideas below.  

\begin{theorem} \label{thm:factor} Given a polynomial $f(x) = a_0x^n + a_1x^{n-1} + \cdots + a_n$ and a positive integer $M \geq 2$, let all the leading coefficients $h_0, \ \ldots, \ h_n$ of the polynomials $f_0, \ \ldots, \ f_n$ obtained by applying the generalized Euclidean algorithm with step $M$ be nonzero. Then the  matrix $\widetilde{H}_M(f)$ factors as follows:
\begin{equation}\label{ku10}
\widetilde{H}_M(f) = J(c_1) \cdots J(c_n)\widetilde{H}_M(a_n),
\end{equation}
where $$J(c_i) = \{{\zeta}_{ij}\}_{i,j =1}^{\infty},$$
$$\zeta_{ij} = \left\{\begin{array}{lll}
c_i =\dfrac{h_{i-1}}{h_i} & \mbox{if} \qquad i = j = Mk+1, & k \geq 0  ;\\[1mm]
1& \mbox{if} \qquad j = i+1, & k \geq 0 ;\\[1mm]
0 & \mbox{otherwise.}\end{array}\right. $$
\end{theorem}

\begin{proof} We will prove formula \eqref{ku10} by induction on $n \eqbd \deg(f)$. For $\deg(f) = 1$, we have $f = a_0(x)+a_1$, $h_0 = a_0$, $h_1 = a_1$, and we obtain
 $$\begin{pmatrix}a_0 & 0 & 0 & \cdots \\
0 & 0 & 0 & \cdots \\
\vdots & \vdots & \vdots & \ddots \\
0 & a_1 & 0 & \cdots \\
0 & a_0 & 0& \cdots \\
\vdots & \vdots & \vdots & \ddots \\
 \end{pmatrix} = \begin{pmatrix} \frac{a_0}{a_1} & 1 & 0 &  \cdots & 0 & 0 & \cdots \\
0 & 0 & 1 & \cdots & 0 & 0 & \cdots \\
\vdots & \vdots & \vdots & \vdots & \vdots & \vdots & \ddots \\
0 & 0 & 0 & \cdots & 1 & 0 & \cdots \\
0 & 0 & 0 & \cdots & \frac{a_0}{a_1} & 1 & \cdots \\
\vdots & \vdots & \vdots & \vdots & \vdots & \vdots & \ddots \\ 
\end{pmatrix}\begin{pmatrix}a_1 & 0 & 0 & \cdots \\
0 & 0 & 0 & \cdots \\
\vdots & \vdots & \vdots & \ddots \\
0 & 0 & 0 & \cdots \\
0 & a_1 & 0 & \cdots \\
\vdots & \vdots & \vdots & \ddots \\
 \end{pmatrix},$$
 i.e., $$ \widetilde{H}_M(f) = J(c_1)\widetilde{H}_M(a_1),$$
where $c_1 = \dfrac{a_0}{a_1} = \dfrac{h_1}{h_0}$. Hence formula \eqref{ku10} holds for $n=1$.

Let us assume that the induction hypothesis holds for $n-1$.  Represent the polynomial $f(x)$ in the usual form    $f(x) = f_0(x) + f_1(x) + \cdots + f_{M-1}(x),$ where each $f_i$ picks up the terms 
with coefficients that are $i$ mod $M$. Run the 
generalized Euclidean algorithm with step $M$
to generate a sequence $f_M$, $\ldots$, $f_n$. The last polynomial $f_n$ has degree zero.
Now define polynomials $F_0, \ \ldots, \ F_n$ as follows:
\begin{align*} F_0(x)  & = \ f_0(x) + f_1(x) + \cdots + f_{M-1}(x) (= f(x)); \\
F_1(x) & = \ f_1(x) + f_2(x) + \cdots + f_M(x);  \\
\cdots \cdots  &  \cdots \cdots   \cdots \cdots \cdots \cdots \cdots \cdots \\
F_{n-(M-1)}(x) & = \ f_{n-(M-1)}(x) + f_{n-(M-2)}(x) + \cdots + f_{n}(x);  \\
F_{n-(M-2)}(x) & =  \ f_{n-(M-2)}(x) + f_{n-(M-3)}(x) + \cdots + f_{n}(x); \\
\cdots \cdots  &   \cdots \cdots \cdots \cdots \cdots \cdots  \cdots \cdots \\
F_n(x) & = \ f_n(x). \end{align*}
Since all the leading coefficients $h_0, \ \ldots, \ h_n$ of the polynomials $f_0, \ \ldots, \ f_n$ are nonzero, we have $\deg F_{i} = n-i$, $i = 0, \ \ldots, \ n$.
Since  $f_M(x) = f_0(x) - \dfrac{h_0}{h_1}xf_1(x) $,  note that 
 $$F_1(x) = f_1(x) + f_2(x) + \cdots + (f_0(x) - \frac{h_0}{h_1}xf_1(x)).$$
We now verify directly that
$\widetilde{H}_M(F_0) = J(c_1)\widetilde{H}_M(F_1),$ where $c_1 = \dfrac{a_0}{a_1} = \dfrac{h_0}{h_1}$. Since $\deg(F_1) = n-1$, $F_1$ satisfies the induction hypothesis, so the matrix $\widetilde{H}_M(F_1)$ factors further as claimed. This proves formula~\eqref{ku10} and the theorem. \hfill \qed \end{proof} \medskip

\noindent
{\bf Remark.} Applying Corollary~\ref{cor31}, we obtain that, for $i = 1,  \ldots,  n$,
$$h_i =\dfrac{H_M(k+1,r)}{H_M(k+1,r-1)}, \quad {\rm where}\quad 
r = \left\lceil \dfrac{i}{M-1}\right\rceil , \qquad k = r(M-1) - i.$$
Thus, if $\left\lceil \dfrac{i}{M-1}\right\rceil = \left\lceil \dfrac{i-1}{M-1}\right\rceil$ then
$c_i = \dfrac{h_{i-1}}{h_i} = \dfrac{H_M(k+2,r)H_M(k+1,r-1)}{H_M(k+2,r-1)H_M(k+1,r)}$. The case $\left\lceil \dfrac{i}{M-1}\right\rceil \neq \left\lceil \dfrac{i-1}{M-1}\right\rceil$ is possible if and only if $i = l(M-1)+1$, where $l$ is some nonnegative integer. Then $\left\lceil \dfrac{i-1}{M-1}\right\rceil = \left\lceil \dfrac{i}{M-1}\right\rceil - 1$, and 
$$h_{i-1} =\dfrac{H_M(k - M + 3,r-1)}{H_M(k-M+3,r-2)},\quad
c_i = \dfrac{h_{i-1}}{h_i} = \dfrac{H_M(k-M+3,r-1)H_M(k+1,r-1)}{H_M(k-M+3,r-2)H_M(k+1,r)}.$$ 

Just in case of the ordinary Hurwitz matrix, Theorem~\ref{thm:factor} yields yet another proof of Theorem~\ref{EuHur}. We restate this theorem slightly to include the matrix $\widetilde{H}_M(f)$ as well.

\begin{theorem}
Given a polynomial $f(x) = a_0x^n + a_1x^{n-1} + \cdots + a_n$ and a positive integer $M \geq 2$, let all the leading coefficients $h_0, \ \ldots, \ h_n$ of the polynomials $f_0, \ \ldots, \ f_n$ obtained by applying the generalized Euclidean algorithm with step $M$ be positive. Then both matrices $H_M(f)$ and $\widetilde{H}_M(f)$ are totally nonnegative.
\end{theorem}

\begin{proof} If all coefficients $h_i$ are positive, then formula~\eqref{ku10} holds, and all factors in that formula are totally nonnegative by inspection. A product of totally nonnegative matrices is totally nonnegative by the Cauchy-Binet formula, so the product $\widetilde{H}_M$ is totally nonnegative. Since $H_M$ is a submatrix of $\widetilde{H}_M$ (and, beautifully, vice versa), we conclude that $H_M$ is totally nonnegative as well. \hfill \qed
\end{proof}


\section{Totally nonnegative submatrices of generalized Hurwitz matrices and zero localization}\label{sec:localization}

We now ask another 'natural' question: is the generalized Hurwitz matrix $H_M$ necessarily totally nonnegative whenever the corresponding (ordinary) Hurwitz matrices $H_M^{(ij)}$ are totally nonnegative  for all pairs $(i,j)$, $0 \leq i < j \leq M-1$?  The answer to this question happens to be negative.

We first recall that the total nonnegativity of the ordinary Hurwitz matrix $H_2(f)$ follows from the positivity of its leading principal minors up to order $n=\deg(f)$ by the classical Hurwitz theorem. 
We will see now that even this (slightly) stronger condition, namely, the positivity of the relevant principal minors of all matrices $H_M^{(ij)}$, $0 \leq i < j \leq M-1$, does not imply that $H_M(f)$ is totally nonnegative if $M\neq 2$. Here is an example.

 \begin{example}Let $n=6$, $M=3$ and $f(x) = x^6 + 3x^5 + 9x^4+ \frac{3}{2}x^3+2x^2 +x +\frac{1}{9}$. Then \begin{equation*}f_0(x) = x^6 + \frac{3}{2} x^3+ \frac{1}{9},\quad f_1(x) = 3x^5 + 2x^2,\quad f_2(x) = 9x^4 + x.\end{equation*}
The leading principal minors of the matrix 
\begin{equation*}
H_3^{(01)} = \begin{pmatrix} 3 & 2 & 0 & 0 & \cdots \\
1 & \frac{3}{2} & \frac{1}{9} & 0  & \cdots \\
0 & 3 & 2 & 0 & \cdots\\
 0 & 1 & \frac{3}{2} & \frac{1}{9} & \cdots \\
\vdots & \vdots & \vdots & \vdots & \ddots \\
\end{pmatrix}
 \end{equation*} up to order $4$ are $3$, $5/2$, $4$, and $4/9$. 
The leading principal minors of the matrix 
\begin{equation*}H_3^{(02)} = \begin{pmatrix} 9 & 1 & 0 & 0 & \cdots \\
1 & \frac{3}{2} & \frac{1}{9} & 0 & \cdots \\
0 & 9 & 1 & 0 & \cdots \\
 0 & 1 & \frac{3}{2} & \frac{1}{9} & \cdots \\
\vdots & \vdots & \vdots & \vdots & \ddots 
\end{pmatrix}
 \end{equation*} are $9$, $25/2$, $7/2$, and $7/18$. 
Finally, the leading principal minors of 
\begin{equation*}H_3^{(12)} = \begin{pmatrix} 9 & 1 & 0 & \cdots \\
3 & 2 & 0 & \cdots \\
 0 & 9 & 1 & \cdots  \\
\vdots & \vdots & \vdots &  \ddots
\end{pmatrix}
 \end{equation*} up to order $3$ are $9$, $15$, and $15$.

However, 
$$H_3 \begin{pmatrix} 2 & 3 & 4 \\
1 & 2 & 3 \end{pmatrix} = \begin{vmatrix}3 & 2 & 0 \\
1 & \frac{3}{2} & \frac{1}{9} \\
0 & 9 & 1  \end{vmatrix} = - \dfrac{1}{2} < 0.$$
\end{example}

Nevertheless,  is all matrices $H_M^{(ij)}$ have positive leading principal minors for 
all pairs $(i,j)$, $0 \leq i < j \leq M-1$, then all such matrices are totally nonnegative, and
this was the only condition (rather than the total nonnegativity of the whole matrix)
that we used while proving Theorem~\ref{coolthm}. So, this theorem can also be stated as follows:

\begin{theorem}\label{variationthm} Let $f(x) = a_0x^n + a_1x^{n-1} + \cdots + a_n$ be a polynomial of degree $n$ and $M$ be a positive integer satisfying $2 \leq M \leq \lfloor\frac{n}{2}\rfloor+1$. If all infinite matrices $H_M^{(ij)}$ for $i,j = 0, \ \ldots, \ M-1$, $i < j$, have positive leading principal minors, then $f(z)$ cannot have zeros in the cone $|\arg(z)|\leq \frac{\pi}{M}$.
\end{theorem}

We note that Theorem~\ref{coolthm}, or Theorem~\ref{variationthm}, also generalize this result  {\cite[Theorem 4.1]{COT}} for $M=n$:
\begin{theorem}[Cowling, Thron] Let $f(x) = a_0x^n + a_1x^{n-1} + \cdots + a_n$ be a polynomial of degree $n$ with all coefficients $a_0, \ldots, a_n$ positive. Then $f(z) = 0$
 implies $|\arg(z)| > \frac{\pi}{n}$.
\end{theorem}


\section{Sufficient conditions for the total positivity of generalized Hurwitz matrices}\label{sec:obstacles}

One last 'natural' question to address is whether we can prove some kind of a converse to our Main Theorem~\ref{maintheorem}. In other words, whether a real polynomial whose zeros lie outside the cone $|\arg(z)|\leq \frac{\pi}{M}$ must have a totally nonnegative generalized Hurwitz matrix $H_M$.
Unfortunately, this happen to be false with a vengeance: even the stronger condition that the polynomial $f$ be stable does not imply that the generalized Hurwitz matrix $H_M$ must be totally nonnegative for $M\neq 2$. Here is an example.

\begin{example} Consider the polynomial $f(x) = x^5 + x^4 + 5x^3 + 2 x^2 + 4x + \frac{1}{2}$.
In this case we have
$$H_2(f) = \begin{pmatrix}1 & 2 & \frac{1}{2} & 0 & 0 & \cdots \\
1 & 5 & 4 & 0 & 0 & \cdots \\
0 & 1 & 2 & \frac{1}{2} & 0 & \cdots \\
0 & 1 & 5 & 4 & 0 & \cdots \\
0 & 0& 1 & 2 & \frac{1}{2} & \cdots\\
\vdots & \vdots & \vdots & \vdots & \vdots & \ddots \\
\end{pmatrix}.$$
The relevant leading principal minors of $H_2(f)$ (i.e., the special minors~\eqref{GoSo1} for $M=2$) are equal to
$3$, $5/2$, $17/4$, and $17/8$.
Thus,  the polynomial $f(x)$ is stable by the Routh-Hurwitz Theorem~\ref{HurwitzTheorem}.

However,
$$H_3\begin{pmatrix} 1 & 2 \\ 2 & 3 \\ \end{pmatrix} =
\begin{vmatrix} 1 & 4 \\ 1 & 2 \end{vmatrix} = -2 < 0.$$
Thus $H_3(f)$ is not totally nonnegative, and some of the coefficients
$\{h_0, \ \ldots, \ h_n\}$ in the generalized Euclidean algorithm with $M=3$ are negative.\end{example}

But in case $M$ is even and $f$ is stable,  the total positivity of $H_M(f)$ does hold!

\begin{theorem} Let a polynomial $f(x) = a_0x^n + a_1x^{n-1}+ \cdots + a_n$ $(a_0, \ a_1,  \ldots, \ a_n \in {\mathbb R}; \ a_0 > 0)$
be stable. Then, for any $M = 2k$ $(k = 1,  \ldots,  \left\lfloor \dfrac{n}{2}\right\rfloor)$,
its generalized Hurwitz matrix $H_M(f)$ is totally nonnegative.
\end{theorem}
\begin{proof} Let $M = 2k$. Consider all special minors~\eqref{GoSo1}:
\[H_M(j,r)= H_M\begin{pmatrix} j& j+1 & \ldots & j+r-1\\
1& 2 & \ldots & r \end{pmatrix}.\]

One checks directly that
\begin{equation}\label{ku14}
 H_M\begin{pmatrix} j & \ldots & j+r-1\\
1 & \ldots & r \end{pmatrix}  =\left\{\begin{array}{cc}
H_2\begin{pmatrix} 1 & 2 & \ldots & r \\
k-\dfrac{j-1}{2} & 2k - \dfrac{j-1}{2} & \ldots & rk - \dfrac{j-1}{2}\end{pmatrix}
& \mbox{if $j$ is odd} ;\\[10pt]
H_2\begin{pmatrix} 2 & 3 & \ldots & r+1 \\
k-\dfrac{j-2}{2} & 2k - \dfrac{j-2}{2} & \ldots & rk - \dfrac{j-2}{2}\end{pmatrix} & \mbox{if $j$ is even.} \end{array}\right.
\end{equation}
Here $k = \dfrac{M}{2}$, $j = 1, \ldots, M-1$ if $r = 1, 2, \ldots, \left\lceil\dfrac{n}{M-1}\right\rceil - 1$, $j = M-p, \ldots, M-1$ if $r = \left\lceil\dfrac{n}{M-1}\right\rceil$, and $p$ is the remainder after the division of $n$ by $M-1$.

Let us prove that all the minors of $H_2(f)$ defined by Formula \eqref{ku14} are positive. Since $f(x)$ is stable we have that $H_2(f)$ satisfies the conditions of the Routh-Hurwitz Theorem~\ref{HurwitzTheorem} and all its leading principal minors are positive.   Applying Theorem 3.1 of \cite{PINK} or, alternatively, Theorem \ref{Theorem2.2} (c), we get
\begin{equation}\label{ku15}
H_2\begin{pmatrix}i_1 & \ldots & i_r\\
j_1 & \ldots & j_r \end{pmatrix} > 0
\quad \hbox{\rm if and only if} \quad  0 \leq 2j_l - i_l \leq n,  \quad l = 1,  \ldots,  r.
\end{equation}

Now note that
$\min_{1 \leq l \leq r}(2(lk-\dfrac{j-1}{2})-l) =\min_{1 \leq l \leq r}(l(2k-1)- (j-1)) = M-j$
if $j$ is odd.
Likewise, $\min_{1 \leq l \leq r}(2(lk-\dfrac{j-2}{2})-(l+1)) =\min_{1 \leq l \leq r}(l(2k-1)- (j-1)) = M-j$
if $j$ is even.
Next, $\max_{1 \leq l \leq r}(2(lk-\dfrac{j-1}{2})-l) =\max_{1 \leq l \leq r}(l(2k-1)- (j-1)) = r(M-1)-(j-1)$
if $j$ is odd.
Finally, $\max_{1 \leq l \leq r}(2(lk-\dfrac{j-2}{2})-(l+1)) =\max_{1 \leq l \leq r}(l(2k-1)- (j-1)) = r(M-1) - (j-1)$
if $j$ is even.
Since $j = 0, \ \ldots, \ M-1$, and $M-1$ is odd, we have $\min\limits_{1 \leq j \leq M-1}(M-j) = 1$ and $\min\limits_{0 \leq j \leq M-2}(M-j) = 2$.

Since $$\max_{1 \leq r \leq \left\lceil\frac{n}{M-1}\right\rceil}(r(M-1) - (j-1)) = \left\lceil\dfrac{n}{M-1}\right\rceil(M-1) - (j-1),$$
we see that
$$\max_j(\left\lceil\frac{n}{M-1}\right\rceil(M-1) - (j-1)) =\left\lceil\frac{n}{M-1}\right\rceil(M-1) - (M-p-1)= (M-1)(\left\lceil\frac{n}{M-1}\right\rceil - 1) + p = n.$$
Thus all the minors of  the form \eqref{ku14} are mentioned in formulas~\eqref{ku15} among the principal minors of some Hurwitz submatrices above and are therefore guaranteed to be positive. Therefore all special minors of $H_M(f)$ are positive, and hence the matrix $H_M(f)$ is totally nonnegative.
\hfill \qed
\end{proof}

\section*{Acknowledgments} We are grateful to Mikhail Tyaglov for helpful discussions, in particular, for suggesting the idea of the proof of Theorem~\ref{coolthm}, and to Alan Sokal for pointing out its connection with the Aissen-Edrei-Schoenberg-Whitney Theorem. The research leading to these results 
 has received funding from the the European Research Council under the European Union's Seventh
Framework Programme (FP7/2007-2013) / ERC grant agreement ${\rm n}^\circ$ 259173 and from the National Science Foundation under agreement No. DMS-1128155. Any opinions, findings and conclusions or recommendations expressed in this material are those of the authors and do not necessarily reflect the views of the National Science Foundation.

\end{document}